\documentclass[11pt]{amsart}

\usepackage{amssymb}
\usepackage{enumerate}
\usepackage{graphicx}

\newcommand{\foot}[1]{\mbox{}\marginpar{\raggedleft\hspace{0pt}\tiny #1}}

\newcommand{\RR}{\mathbb{R}}
\newcommand{\ZZ}{\mathbb{Z}}

\newcommand{\NN}{\mathbb{N}}

\newcommand{\CC}{\mathbb{C}}

\newcommand{\AAA}{\mathcal{A}}
\newcommand{\BBB}{\mathcal{B}}
\newcommand{\CCC}{\mathcal{C}}

\newcommand{\MMM}{\mathcal{M}}
\newcommand{\HHH}{\mathcal{H}}

\newcommand{\ph}{\varphi}
\newcommand{\eps}{\varepsilon}
\newcommand{\di}{\partial}

\newcommand{\llim}{\varliminf}
\newcommand{\ulim}{\varlimsup}

\newcommand{\ld}{\underline{d}}
\newcommand{\ud}{\overline{d}}
\newcommand{\htop}{h_\mathrm{top}\,}
\newcommand{\abs}[1]{\left\lvert#1\right\rvert}

\DeclareMathOperator{\Crit}{Crit}
\DeclareMathOperator{\diam}{diam}

\renewcommand{\foot}[1]{}

\newcommand{\Faen}{{F_\alpha^{\eps,N}}}
\newcommand{\Fae}{{F_\alpha^\eps}}
\newcommand{\Gaen}{{G_\alpha^{\eps,N}}}
\newcommand{\Gae}{{G_\alpha^\eps}}

\newcommand{\LLL}{\mathcal{L}}
\newcommand{\DDD}{\mathcal{D}}
\newcommand{\PPP}{\mathcal{P}}
\newcommand{\EEE}{\mathcal{E}}
\newcommand{\ZZZ}{\mathbf{Z}}
\newcommand{\BB}{\mathbf{B}}
\newcommand{\Aa}{\textbf{(A)}}

\newcommand{\Kd}[1]{K_{#1}^\DDD}
\newcommand{\Kad}{\Kd{\alpha}}
\newcommand{\Da}{\DDD(\alpha)}
\newcommand{\Ke}[1]{K_{#1}^\EEE}
\newcommand{\Kae}{\Ke{\alpha}}
\newcommand{\Ea}{\EEE(\alpha)}
\newcommand{\Kb}[1]{K_{#1}^\BBB}
\newcommand{\Kab}{\Kb{\alpha}}
\newcommand{\Ba}{\BBB(\alpha)}
\newcommand{\Kl}[1]{K_{#1}^\LLL}
\newcommand{\Kal}{\Kl{\alpha}}
\newcommand{\LDa}{\LLL_D(\alpha)}
\newcommand{\LEa}{\LLL_E(\alpha)}
\newcommand{\Mf}{\MMM^f}
\newcommand{\Mfe}{\MMM^f_E}

\newcommand{\ccph}{\overline{\CCC(\ph)}}

\newcommand{\Af}{\AAA_f}

\newcommand{\Tb}{T_\BBB}
\newcommand{\Te}{T_\EEE}
\newcommand{\Td}{T_\DDD}
\newcommand{\BL}{\BBB^{L_2}}
\newcommand{\TbL}{\Tb^{L_1}}
\newcommand{\TL}{T^{L_1}}
\newcommand{\SL}{S^{L_2}}
\newcommand{\SLL}{(S^{L_2})^{L_1}}
\newcommand{\EL}{\BBB^{L_4}}
\newcommand{\TeL}{\Te^{L_3}}
\newcommand{\DL}{\DDD^{L_4}}
\newcommand{\TdL}{\Td^{L_3}}
\newcommand{\TTL}{T^{L_3}}
\newcommand{\SSL}{S^{L_4}}

\newcommand{\uP}{\overline{CP}}
\newcommand{\lP}{\underline{CP}}

\newcommand{\amin}{\alpha_\mathrm{min}}
\newcommand{\amax}{\alpha_\mathrm{max}}
\newcommand{\llambda}{\underline{\lambda}}
\newcommand{\ulambda}{\overline{\lambda}}
\newcommand{\lowd}{\underline{d}}
\newcommand{\uppd}{\overline{d}}
\newcommand{\lhtop}{\underline{Ch}_\mathrm{top}}
\newcommand{\uhtop}{\overline{Ch}_\mathrm{top}}

\theoremstyle{plain}
\newtheorem{theorem}{Theorem}[section]
\newtheorem{proposition}[theorem]{Proposition}
\newtheorem{corollary}[theorem]{Corollary}
\newtheorem{lemma}[theorem]{Lemma}

\theoremstyle{definition}
\newtheorem{definition}{Definition}[section]

\theoremstyle{remark}
\newtheorem{example}[theorem]{Example}
\newtheorem*{remark}{Remark}

\numberwithin{equation}{section}

\title{Multifractal Formalism Derived from Thermodynamics}
\author{Vaughn Climenhaga}
\address{Department of Mathematics \\ McAllister Building \\ Pennsylvania State University \\ University Park, PA 16802, USA.}
\email{climenha@math.psu.edu}
\urladdr{http://www.math.psu.edu/climenha/}

\begin{document}

\date{\today}
\begin{abstract}
We show that under quite general conditions, various multifractal spectra may be obtained as Legendre transforms of functions $T\colon \RR\to \RR$ arising in the thermodynamic formalism.  We impose minimal requirements on the maps we consider, and obtain partial results for any continuous map $f$ on a compact metric space.  In order to obtain complete results, the primary hypothesis we require is that the functions $T$ be continuously differentiable.  This makes rigorous the general paradigm of reducing questions regarding the multifractal formalism to questions regarding the thermodynamic formalism.  These results hold for a broad class of measurable potentials, which includes (but is not limited to) continuous functions.  We give applications that include most previously known results, as well as some new ones.
\end{abstract}

\thanks{This work is partially supported by NSF grant 0754911.}

\maketitle

\section{Introduction}

\emph{A preliminary announcement (without proofs) of the results in this paper is to appear in \emph{Electronic Research Announcements}.}

\subsection{Overview of multifractal formalism}

The basic elements of the multifractal formalism were first proposed by Halsey \emph{et al} in~\cite{HJKPS86}, where they considered what they referred to as the \emph{dimension spectrum} or the \emph{$f(\alpha)$-spectrum for dimensions}, which characterises an invariant measure $\mu$ for a dynamical system $f\colon X\to X$ in terms of the level sets of the \emph{pointwise dimension}.  The pointwise dimension of $\mu$ at $x$ is defined as
\[
d_\mu(x) = \lim_{\eps\to 0} \frac{\log \mu(B(x,\eps))}{\log\eps},
\]
provided the limit exists, and the level sets are denoted
\[
\Kad = \{ x\in X \mid d_\mu(x) = \alpha \}.
\]
Many measures of interest are \emph{exact-dimensional}; that is, the pointwise dimension is constant $\mu$-almost everywhere.  In particular, this is true of hyperbolic measures (those with non-zero Lyapunov exponents almost everywhere)~\cite{BPS99}.  For an exact-dimensional measure, one of the $\Kad$ has full measure, and the rest have measure $0$, and so we measure the sizes of these sets with the Hausdorff dimension rather than with the measure; in this way we obtain the \emph{dimension spectrum for pointwise dimensions}, which is given by the function
\[
\Da=\dim_H \Kad.
\]

One goal of the multifractal formalism is to show that under certain conditions on $f$ and $\mu$, the function $\DDD$ is in fact analytic and concave on its domain of definition, and is related to the R\'enyi and Hentschel--Procaccia spectra for dimensions by a Legendre transform.  This was done by Rand~\cite{dR89} when $\mu$ is a Gibbs measure on a hyperbolic cookie-cutter (a dynamically defined Cantor set), and by Pesin and Weiss~\cite{PW97} for uniformly hyperbolic conformal maps:  modern expositions of the whole theory for uniformly hyperbolic systems can be found in~\cite{yP98,BPS97,TV00}.  More recently, various non-uniformly hyperbolic systems have been studied in~\cite{kN00,mT08,JR09,IT09b}.

There are other important examples of multifractal spectra; each such spectrum measures the level sets of some local quantity by using a global (dimensional) quantity.  For $\Da$, these roles are played by pointwise dimension and Hausdorff dimensions, respectively; one may also consider spectra defined using other quantities.

For example, one may consider the measure of small balls which are refined dynamically, rather than statically.  That is, rather than $B(x,\eps)$ we consider the Bowen ball of radius $\delta$ and length $n$, given by
\[
B(x,n,\delta) = \{y\in X\mid f^k(y)\in B(f^k(x),\delta) \text{ for } k=0,1,\dots,n \}.
\]
If the map $f$ has some eventual expansion, then the balls $B(x,n,\delta)$ decrease in size, and in measure, as $n$ increases with $\delta$ held fixed.  Just as the rate at which $\mu(B(x,\eps))$ decreases with $\eps$ is the pointwise dimension $d_\mu(x)$, so also the rate at which $\mu(B(x,n,\delta))$ decreases with $n$ is the \emph{local entropy} of $\mu$ at $x$
\[
h_\mu(x) = \lim_{\delta\to 0} \llim_{n\to\infty} -\frac{1}{n} \log \mu(B(x,n,\delta)),
\]
provided the limit exists.  We denote the level sets of the local entropy by
\[
\Kae = \{x\in X \mid h_\mu(x) = \alpha \}.
\]
It was shown by Brin and Katok that if $\mu$ is ergodic, then one of the level sets $\Kae$ has full measure, and the rest have measure $0$~\cite{BK83}; thus we must once again quantify them using a (global) dimensional characteristic.  It turns out to be more natural to measure the size of the sets $\Kae$ with the topological entropy rather than Hausdorff dimension; because these level sets are in general not compact, we must use the definition of topological entropy in the sense of Bowen~\cite{rB73}.  Upon doing so, we obtain the \emph{entropy spectrum for local entropies}
\[
\Ea = \htop (\Kae).
\]
For Gibbs measures on conformal repellers, this spectrum was studied in~\cite{BPS97}.  Takens and Verbitskiy~\cite{TV99} carred out the multifractal analysis in the more general case of expansive maps satisfying a specification property.

The Gibbs property of the measures studied so far is essential, because it relates local scaling quantities of the measure (pointwise dimension or local entropy) to asymptotic statistical properties of a potential function $\ph$.  In fact, the proofs of the known results for both the dimension and entropy spectra contain (at least implicitly) a similar result for the \emph{Birkhoff spectrum}.  Writing the sum of $\ph$ along an orbit as $S_n\ph(x) = \sum_{k=0}^{n-1} \ph(f^k(x))$, the \emph{Birkhoff average} of $\ph$ at $x$ is given by
\[
\ph^+(x) = \lim_{n\to\infty} \frac{1}{n} S_n \ph(x),
\]
provided the limit exists.  The level sets of the Birkhoff averages are
\[
\Kab = \{x\in X \mid \ph^+(x) = \alpha \},
\]
and the Birkhoff ergodic theorem guarantees that for any ergodic measure $\mu$, one of the level sets has full measure, and the rest have measure $0$.  Thus we once again measure their size in terms of topological entropy, and obtain the \emph{entropy spectrum of Birkhoff averages}
\[
\Ba = \htop (\Kab).
\]
In the uniformly hyperbolic setting, results on the Birkhoff spectrum were obtained in~\cite{PW01}, among other places.  More general results, some of which overlap with one of the results in this paper, were recently announced in~\cite{FH10}.

The general scheme tying all these spectra together is as follows.  Given an asymptotic local quantity---pointwise dimension, local entropy, Birkhoff average---we have an associated multifractal decomposition into level sets of this quantity.  These level sets are then measured using a global dimensional quantity---Hausdorff dimension or topological entropy.  This defines a multifractal spectrum, which associates to each real number $\alpha$ the dimension of the level set corresponding to $\alpha$.  

Following this general outline, each of the above spectra could also be defined using the alternate global dimensional quantity.  That is, we could define the \emph{entropy spectrum for pointwise dimensions} by
\[
\DDD_E(\alpha) = \htop(\Kad),
\]
and similarly for the \emph{dimension spectrum for local entropies} and the \emph{dimension spectrum for Birkhoff averages}.  It turns out that these \emph{mixed multifractal spectra} are harder to deal with than the ones we have defined so far; see~\cite{BS01} for further details.  We will restrict our attention to the spectra for which the local and global quantities are naturally related, and will generally simply refer to the \emph{entropy spectrum}, the \emph{dimension spectrum}, and the \emph{Birkhoff spectrum}.

We will see that the Birkhoff spectrum provides a simpler setting for arguments which also apply to the dimension and entropy spectra; it is also of interest in its own right, having applications to the theory of large deviations~\cite{PW01, BR87}.

One important example of a Birkhoff spectrum is worth noting.  In the particular case where $f$ is a conformal map and $\ph(x)=\log \|Df(x)\|$, the Birkhoff averages coincide with the Lyapunov exponents:  $\lambda(x) = \ph^+(x)$.  In this case we will also denote the level sets by
\[
\Kal = \{ x\in X \mid \lambda(x) = \alpha \};
\]
it turns out that we are able to examine not only the \emph{entropy spectrum for Lyapunov exponents}
\[
\LEa = \htop \Kal,
\]
but also the \emph{dimension spectrum for Lyapunov exponents}
\[
\LDa = \dim_H \Kal,
\]
by using a generalisation of Bowen's equation to non-compact sets~\cite{BS00,vC09a}.  We may refer to either $\LEa$ or $\LDa$ as the \emph{Lyapunov spectrum}.  It is often the case that the ``interesting'' dynamics takes place on a repeller which has Lebesgue measure zero---the Lyapunov spectrum provides information on how quickly the trajectories of nearby points escape from a neighbourhood of the repeller~\cite{BR87}.

Taken together, the various multifractal spectra provide a great deal of information about the map $f$.  In fact, certain classes of systems are known to exhibit \emph{multifractal rigidity}, in which a finite number of multifractal spectra completely characterise a map~\cite{BPS97}.


\subsection{General description of results}

Direct computation (numerical or otherwise) of the various multifractal spectra is quite difficult.  In the first place, in order to determine the level sets $K_\alpha$, one needs to first compute the asymptotic quantity (Birkhoff average, pointwise dimension, local entropy) at \emph{every} point of $X$.  Even if this is accomplished, it still remains to compute the (Bowen) topological entropy or Hausdorff dimension of $K_\alpha$ for every value of $\alpha$.  Because this quantity is defined as a critical point, rather than as a growth rate, it is more difficult to compute than the (capacity) topological entropy or the box dimension.  (These latter quantities are of little use in analysing the level sets $K_\alpha$ since they assign the same value to a set and to its closure, and the level sets $K_\alpha$ are dense in many natural situations.)

Rather than a direct frontal assault, then, the most successful method for analysing multifractal spectra has been to relate them to certain thermodynamic functions via the Legendre transform.  These thermodynamic functions, which are given in terms of the topological pressure, can be computed more easily than the multifractal spectra, as they are given in terms of the growth rates of a family of partition functions.

This approach goes back to~\cite{dR89} (the Legendre transform appeared already in~\cite{HJKPS86}, but in terms of the Hentschel--Procaccia and R\'enyi spectra, not in terms of the topological pressure).  To date, the general strategy informed by this philosophy has been as follows:
\begin{enumerate}[(1)]
\item Fix a specific class of systems---uniformly hyperbolic maps, conformal repellers, parabolic rational maps, Manneville--Pomeau maps, multimodal interval maps, etc.
\item Using tools specific to that class of systems (Markov partitions, specification, inducing schemes), establish thermodynamic results---existence and uniqueness of equilibrium states, differentiability of the pressure function, etc.
\item Using these thermodynamic results \emph{together with the original toolkit}, study the multifractal spectra, and show that they can be given in terms of the Legendre transform of various pressure functions.
\end{enumerate}

Despite the success of this approach for a number of different classes of systems, there do not appear to be any extant rigorous results which apply to general continuous maps and arbitrary potentials (but see the remark below concerning~\cite{FH10}).  Such results would give information about the multifractal analysis in settings far beyond those already considered; they would also establish the multifractal analysis as a direct corollary of the thermodynamic formalism, rendering Step (3) above automatic, and eliminating the need for the use of a specific toolkit to study the multifractal formalism itself.

The results of this paper are a step in this direction.  Not only do we obtain results that apply to general continuous maps regarding which nothing had been known, but the results described below also give alternate proofs of most previously known multifractal results, which are in some cases more direct than the original proofs.

We obtain our strongest result for the Birkhoff spectrum $\Ba$.  This result is given in Theorem~\ref{thm:birkhoff}, which applies to continuous maps $f\colon X\to X$ and to functions $\ph\colon X\to \RR$ which lie in a certain class $\Af$; this class contains, but is not limited to, the space of all continuous functions.  For such maps and functions, we show that the function $\Tb\colon q\mapsto P(q\ph)$, where $P$ is the pressure, is the Legendre transform of $\Ba$, \textbf{\emph{without any further restrictions on $f$ and $\ph$}}.  Furthermore, we show that $\Ba$ is the Legendre transform of $\Tb$, completing the multifractal formalism, \textbf{\emph{provided $\Tb$ is continuously differentiable and equilibrium measures exist}}.  If the hypotheses on $\Tb$ only hold for certain values of $q$, we still obtain a partial result on $\Ba$ for the corresponding values of $\alpha$.

\begin{remark}
After this paper was completed, the author was made aware of recent results announced by Feng and Huang in~\cite{FH10}, which deal with asymptotically sub-additive sequences of potentials, and which include Theorem~\ref{thm:birkhoff} for continuous potentials $\ph$ as a special case (however, they do not consider any of the dimension spectra).  Many of the methods of proof are similar, and it appears as though the other results in this paper could also be extended to the non-additive case they consider.

We observe that due to their definition of pressure, which only applies to functions $\ph$ such that $e^{\ph(x)}$ is continuous, their results do not apply to the discontinuous potentials in $\Af$, nor to the more general class of bounded measurable potentials for which we obtain partial results (see below).  To the best of the author's knowledge, the present results are the first rigorous multifractal results for general discontinuous potentials.
\end{remark}

Theorem~\ref{thm:birkhoff} gives an alternate (and more direct) proof of the multifractal formalism for the Birkhoff spectrum of a H\"older continuous potential function and a uniformly hyperbolic system, which was first established by Pesin and Weiss~\cite{PW01}.  It can also be applied to non-uniformly hyperbolic systems; in addition to some systems that have already been studied, we describe in Section~\ref{sec:app} a class of systems studied by Varandas and Viana~\cite{VV08} to which Theorem~\ref{thm:birkhoff} can be applied.  Proposition~\ref{prop:VV} gives multifractal results for these systems; these results appear to be completely new.

As stated, Theorem~\ref{thm:birkhoff} does not deal with phase transitions---that is, points at which the pressure function is non-differentiable.  Such points correspond (via the Legendre transform) to intervals over which the Birkhoff spectrum is affine (if the multifractal formalism holds).  In Theorem~\ref{thm:phase}, we give slightly stronger conditions on the map $f$, which are still fundamentally thermodynamic in nature, under which we can establish the complete multifractal formalism even in the presence of phase transitions.

It is often the case that thermodynamic considerations demonstrate the existence of a \emph{unique} equilibrium state for certain potentials.  In Proposition~\ref{prop:unique-works}, we show that if the entropy function is upper semi-continuous, then uniqueness of the equilibrium state implies differentiability of the pressure function and allows us to apply Theorem~\ref{thm:birkhoff}.  However, Example~\ref{eg:vw} shows that there are systems for which the pressure function is differentiable, and hence Theorem~\ref{thm:birkhoff} can be applied, even though the equilibrium state is non-unique.


One would like to understand for which classes of discontinuous potentials the multifractal formalism holds.  Things work well for $\ph\in \Af$ because the weak* topology is the same at $f$-invariant measures whether we consider continuous test functions or test functions in $\Af$.

Beyond this class of potentials, things are more delicate.  We consider general measurable potentials that are bounded above and below, and while we do not obtain results for all values of $\alpha$, we do obtain in Theorem~\ref{thm:high-entropy} complete results for those values of $\alpha$ at which $\TbL$ is larger than the topological entropy of the closure of the set of discontinuities of $\ph$, and for the corresponding values of $q$.

Ideally, we would be able to include \emph{unbounded} potentials in these results.  In particular, we would like to be able to consider the geometric potential $\ph(x) = -\log |f'(x)|$ for a multimodal map $f$; the presence of critical points leads to singularities of $\ph$, and so $\ph$ is not bounded above.  Theorem~\ref{thm:singularity} shows that the results of Theorem~\ref{thm:birkhoff} still hold for $q\leq 0$ (that is, values of $q$ such that $q\ph$ is bounded above) and for the corresponding values of $\alpha$.  The question of what happens for $q>0$ is more delicate and remains open.

In Section~\ref{sec:conformal}, we use a non-uniform version of Bowen's equation~\cite{vC09a} to give a result for the Lyapunov spectrum $\LDa$ in the case where $f$ is a conformal map without critical points, which satisfies some asymptotic expansivity properties.

In order to obtain results on the spectra $\Ea$ and $\Da$, for which the corresponding local quantities ($d_\mu(x)$ and $h_\mu(x)$) are defined in terms of an invariant measure $\mu$, we need some relationship between $\mu$ and a potential function $\ph$.  This is given by the assumption that $\mu$ is a \emph{weak Gibbs measure} for $\ph$; we observe that there are several cases in which weak Gibbs measures (of one definition or another) are known to exist~\cite{mY00,mK01,FO03,VV08,JR09}.

For such measures, we will see that the level sets $\Kae$ are determined by the level sets $\Kab$, and hence we obtain Theorem~\ref{thm:entropy}, which gives the corresponding result for the entropy spectrum $\Ea$ of a Gibbs measure, and follows from Theorem~\ref{thm:birkhoff}.  Writing $\ph_1=\ph-P(\ph)$, we find $\Ea$ as the Legendre transform of the function $\Te\colon q\mapsto P(-q\ph_1)$, \textbf{\emph{provided $\Te$ is continuously differentiable and equilibrium measures exist}}.  

Theorem~\ref{thm:dimension} deals with the dimension spectrum $\Da$ in the case where $f$ is conformal without critical points and $\mu$ is a weak Gibbs measure for a continuous potential $\ph$.  Passing to $\ph_1$ so that $P(\ph_1)=0$, we follow Pesin and Weiss~\cite{PW97}, and define a family of potential functions $\ph_q$ by
\begin{equation}\label{eqn:implicit}
\ph_q(x) = -\Td(q) \log \|Df(x)\| + q\ph_1(x),
\end{equation}
with $\Td(q)$ chosen so that $P(\ph_q)=0$.  Under mild expansivity conditions on $f$, we show that the implicitly defined function $\Td(q)$ is the Legendre transform of the dimension spectrum $\Da$, \textbf{\emph{without any further conditions on $f$ or $\ph$}}.  Furthermore, we show that $\Da$ is the Legendre transform of $\Td(q)$, completing the multifractal formalism, \textbf{\emph{provided $\Td$ is continuously differentiable and equilibrium measures exist}}.

Results for all of the above spectra have already been known in specific cases.  However, the present results differ from previous work in that their proofs do not use properties of the map $f$ that are specific to a particular class, but rather rely on thermodynamic results.  This is particularly true of Theorem~\ref{thm:birkhoff}, which requires nothing at all of $f$ besides continuity.  We also observe that the requirement of conformality in~\eqref{eqn:dimlyap} and Theorem~\ref{thm:dimension} is somehow unavoidable if we wish to use the any of the standard definitions of pressure; for a non-conformal map, one would need to consider a non-additive version of the pressure~\cite{lB96,FH10}, and it is not clear what implicit definition for $\Td$ should replace~\eqref{eqn:implicit}.

In Sections~\ref{sec:rmk} and~\ref{sec:app}, we make various general remarks concerning the results and their applications to both known and new examples.  Sections~\ref{sec:prep} through~\ref{sec:last} contain the proofs.

\emph{Acknowledgements.}  Many thanks are due to my advisor, Yakov Pesin, for the initial suggestion to pursue this approach, and for much guidance and encouragement along the way.  I would also like to thank Van Cyr, Katrin Gelfert, Stefano Luzzatto, Omri Sarig, Sam Senti, and Mike Todd for helpful conversations as this work took on its present form.

\section{Definitions and results for Birkhoff spectrum}

Throughout this section, we fix a compact metric space $X$, a continuous map $f\colon X\to X$, and a Borel measurable potential function $\ph\colon X\to \RR$.

To fix notation, we recall the definition of Hausdorff dimension.

\begin{definition}
Given $Z\subset X$ and $\eps>0$, let $\DDD(Z,\eps)$ denote the collection of countable open covers $\{ U_i \}_{i=1}^\infty$ of $Z$ for which $\diam U_i\leq \eps$ for all $i$.  For each $s\geq 0$, consider the set functions
\begin{align}
\label{eqn:mHse}
m_H(Z,s,\eps) &= \inf_{\DDD(Z,\eps)} \sum_{U_i} (\diam U_i)^s, \\
\label{eqn:mHs}
m_H(Z,s) &= \lim_{\eps\to 0} m_H(Z,s,\eps).
\end{align}
The \emph{Hausdorff dimension} of $Z$ is
\[
\dim_H Z = \inf \{ s>0 \mid m_H(Z,s)=0 \} = \sup \{ s>0 \mid m_H(Z,s)=\infty \}.
\]
It is straightforward to show that $m_H(Z,s)=\infty$ for all $s<\dim_H Z$, and that $m_H(Z,s)=0$ for all $s>\dim_H Z$.
\end{definition}

An analogous definition of topological entropy was given by Bowen~\cite{rB73}, establishing it as another dimensional characteristic.

\begin{definition}
Given $Z\subset X$, $\delta>0$, and $N\in\NN$, let $\PPP(Z,N,\delta)$ denote the collection of countable sets $\{ (x_i,n_i) \}_{i=1}^\infty \subset Z\times\NN$ such that $\{B(x_i,n_i,\delta)\}$ covers $Z$ and $n_i\geq N$ for all $i$.  For each $s\in\RR$, consider the set functions
\begin{align}
\label{eqn:mhNd}
m_h(Z,s,N,\delta) &= \inf_{\PPP(Z,N,\delta)} \sum_{(x_i,n_i)} e^{-n_i s}, \\
\label{eqn:mhd}
m_h(Z,s,\delta) &= \lim_{N\to\infty} m_h(Z,s,N,\delta),
\end{align}
and put
\[
\htop(Z,\delta) = \inf \{ s>0 \mid m_h(Z,s,\delta) = 0 \} = \sup \{ s>0 \mid m_h(Z,s,\delta) = \infty \}.
\]
As with Hausdorff dimension, we get $m_h(Z)=\infty$ for $s<\htop(Z,\delta)$, and $m_h(Z)=0$ for $s>\htop(Z,\delta)$.  The \emph{topological entropy} of $f$ on $Z$ is
\[
\htop(Z) = \lim_{\delta\to 0} \htop(Z,\delta).
\]
\end{definition}

If we replace the quantity $e^{-n_i s}$ in~\eqref{eqn:mhNd} with $e^{n_i s + S_{n_i} \ph(x_i)}$, the definition above gives us not the topological entropy but the topological \emph{pressure} $P_Z(\ph)$, introduced in this form by Pesin and Pitskel' in~\cite{PP84} (although the version of pressure we will discuss below dates back to Ruelle~\cite{dR73} and Bowen~\cite{rB75b}).  All three of these quantities (Hausdorff dimension, entropy, and pressure) are defined as critical points and have certain important properties common to a broad class of Carath\'eodory dimension characteristics (see~\cite{yP98} for details).  We will use two of these repeatedly, so they are worth mentioning here:  in the first place, given any countable family of sets $Z_i\subset X$, we have
\[
\dim_H \left(\bigcup_i Z_i\right) = \sup_i \dim_H Z_i,
\]
and similarly for $\htop Z$ and $P_Z(\ph)$.  Furthermore, all of these quantities can be bounded above in terms of a corresponding \emph{capacity}; for Hausdorff dimension, the corresponding capacity is the lower box dimension.  We recall the definitions of the analogues for entropy and pressure (see~\cite{yP98}).

\begin{definition}
A set $E\subset Z$ is $(n,\delta)$-spanning if $Z \subset \bigcup_{x\in E} B(x,n,\delta)$.  The \emph{(lower) capacity topological entropy} $\lhtop(Z)$ is the lower asymptotic growth rate of the minimal cardinality of an $(n,\delta)$-spanning set in $Z$.  More precisely, if $P_n^\delta$ is the minimal cardinality of such a set, then
\begin{align}
\label{eqn:lhZd}
\lhtop(Z,\delta) &= \llim_{n\to\infty} \frac 1n \log P_n^\delta, \\
\label{eqn:lhZ}
\lhtop(Z) &= \lim_{\delta\to 0} \lhtop(Z,\delta).
\end{align}
A similar definition taking the upper limit gives us $\uhtop(Z)$.

In the proof of Theorem~\ref{thm:dimension}, we will also need the notion of \emph{capacity topological pressure}, whose definition we recall here.  Fix a potential $\ph\colon X\to \RR$ and a subset $Z\subset X$.  For every $n\in \NN$, $\delta>0$, let $E_n$ be a minimal $(n,\delta)$-spanning set: then the lower capacity topological pressure of $\ph$ on $Z$ is given by
\begin{align}
\label{eqn:lPZd}
\lP_Z(\ph,\delta) &= \llim_{n\to\infty} \frac 1n \sum_{x\in E_n} e^{S_n \ph(x)}, \\
\label{eqn:lPZ}
\lP_Z(\ph) &= \lim_{\delta\to 0} \lP_Z(\ph,\delta).
\end{align}
We have a corresponding definition of $\uP_Z(\ph)$.  In the case $\ph=0$, these reduce to $\lhtop(Z)$ and $\uhtop(Z)$, respectively.

Elementary arguments given in~\cite{pW75} show that we can also use maximal $(n,\delta)$-separated sets in the above definitions, and we will occasionally do so.
\end{definition}

We observe that the definitions given above differ slightly from the definitions in~\cite{yP98}.  For a proof that both sets of definitions yield the same quantities when the potential $\ph$ is continuous, see~\cite[Proposition~4.1]{vC09a}.

In general, we have the following relationship between the three pressures~\cite[(11.9)]{yP98}:
\begin{equation}\label{eqn:pressures}
P_Z(\ph) \leq \lP_Z(\ph) \leq \uP_Z(\ph).
\end{equation}
If $Z$ is compact and $f$-invariant (for example, if $Z=X$), then we have equality in~\eqref{eqn:pressures}, and the variational principle relates the common quantity to the following definition, which we will use to state our thermodynamic requirements.

\begin{definition}
Let $\MMM(X)$ be the set of all Borel probability measures on $X$, and denote by $\Mf(X)$ the set of $f$-invariant measures in $\MMM(X)$.  

Given $\mu\in\Mf(X)$, write $h(\mu)$ for the measure theoretic entropy of $\mu$.  The \emph{(variational) pressure} of $\ph$ is
\begin{equation}\label{eqn:pressure}
P^*(\ph) = \sup \left\{ h(\mu) + \int \ph\,d\mu \,\Big|\, \mu\in \Mf(X) \right\}.
\end{equation}
If $Z\subset X$ is compact and $f$-invariant, we will write the pressure on $Z$ as
\[
P_Z^*(\ph) = \sup \left\{ h(\mu) + \int \ph\,d\mu \,\Big|\, \mu\in \Mf(Z) \right\},
\]
where $\Mf(Z) = \{ \mu \in \Mf(X) \mid \mu(Z) = 1 \}$.

Let $\Mfe(X)$ be the set of all ergodic measures in $\Mf(X)$.  It follows using the ergodic decomposition that~\eqref{eqn:pressure} is equivalent to
\[
P^*(\ph) = \sup \left\{ h(\mu) + \int\ph\,d\mu \,\Big|\, \mu\in \Mfe(X) \right\}.
\]
A measure $\nu\in\Mf(X)$ is an \emph{equilibrium state} for the potential $\ph$ if it achieves this supremum; that is, if
\[
P^*(\ph) = h(\nu) + \int\ph\,d\nu.
\]
Every equilibrium state is a convex combination of ergodic equilibrium states.
\end{definition}

As is customary in multifractal formalism, we use the Legendre transform in the following slightly non-standard form.

\begin{definition}
Recall that a function $T\colon \RR\to [-\infty, +\infty]$ is convex if
\begin{equation}\label{eqn:cvx}
T(aq + (1-a)q') \leq aT(q) + (1-a) T(q')
\end{equation}
for all $0\leq a\leq 1$ and $q,q'\in \RR$.  Given a convex function $T$, the \emph{Legendre transform} of $T$ is
\begin{equation}\label{eqn:TL}
\TL(\alpha) = \inf_{q\in \RR} (T(q) - q\alpha).
\end{equation}
Given a concave function $S\colon \RR\to [-\infty,+\infty]$ (for which the inequality in~\eqref{eqn:cvx} is reversed), the Legendre transform of $S$ is
\begin{equation}\label{eqn:SL}
\SL(q) = \sup_{\alpha\in \RR} (S(\alpha) + q\alpha).
\end{equation}
\end{definition}

The Legendre transform of a convex function is concave, and vice versa.  Furthermore, the Legendre transform is self-dual:  if $T$ is convex and $\TL = S$, then $\SL = T$.  Similarly, if $S$ is concave and $\SL = T$, then $\TL = S$.  

In what follows, we will consider situations in which the function $T$ is known to be convex (being given in terms of the pressure function), but the function $S$ is one of the multifractal spectra, about which we have no \emph{a priori} knowledge.  Observe that the Legendre transform of such a function $S$ can still be defined by~\eqref{eqn:SL}, but in this case we lose duality; in its place we get the statement that $\SLL$ is the concave hull of $S$, the smallest concave function bounded below by $S$.

Observe also that if $S(x)\geq 0$ for every $x\in \RR$, then $\SL$ is infinite everywhere.  Thus for purposes of defining the various multifractal spectra, we adopt the (non-standard) convention that $\htop \emptyset = \dim_H \emptyset = -\infty$.

We recall that if $T$ is known to be convex, then left and right derivatives exist at every point where $T$ is finite; we will denote these by
\[
D^- T(q) = \lim_{q'\to q^-} \frac{T(q) - T(q')}{q-q'}, \qquad
D^+ T(q) = \lim_{q'\to q^+} \frac{T(q') - T(q)}{q'-q}.
\]
Existence follows from monotonicity of the slopes of the secant lines.  Given a convex function $T$, define a map from $\RR$ to closed intervals in $\RR$ by $A(q) = [D^- T(q), D^+ T(q)]$.  Extend this in the natural way to a map from subsets of $\RR$ to subsets of $\RR$; we will again denote this map by $A$.  This map has the following useful property:  given any set $I_Q\subset \RR$ and $\alpha\in A(I_Q)$, we have
\[
\TL(\alpha) = \sup_{q\in I_Q} (T(\alpha) + q\alpha).
\]
This will be important for us in settings where we only have partial information about the functions $T$ and $S$.  We will also make use of a map in the other direction:  given a set $I_A \subset \RR$ (in the domain of $S$), we denote the set of corresponding values of $q$ by
\[
Q(I_A) = \{q\in \RR \mid A(q)\cap I_A \neq \emptyset \}.
\]
In particular, if $\alpha = T'(q)$, then $\alpha = A(q)$, and if $q=-S'(\alpha)$, then $q = Q(\alpha)$.  If $(q_1,q_2)$ is an interval on which $T$ is affine, then $A((q_1,q_2))$ is the slope of $T$ on that interval; furthermore, $\TL$ has a point of non-differentiability at $A((q_1,q_2))$.

In the results below, it will sometimes be important to know whether or not $T$ is differentiable.  A standard cardinality argument shows that $D^-T(q) = D^+T(q)$ at all but countably many values of $q$; however, the values of $q$ at which differentiability fails may \emph{a priori} be dense in $\RR$.

Our most general result gives the following function as the Legendre transform of the Birkhoff spectrum:
\begin{equation}\label{eqn:Tb}
\Tb(q) = P^*(q\ph),
\end{equation}
Note that the function $\Tb$ is convex; even before we establish that $\Tb$ is the Legendre transform of $\Ba$, convexity follows immediately from the definition of variational pressure as a supremum and the fact that for every $\mu\in \Mf(X)$, the function $q\mapsto h(\mu) + \int q\ph\,d\mu$ is linear.  

Finally, before stating the general result, we describe the class of functions to which it applies.  Given a function $\ph\colon X\to \RR$, let $\CCC(\ph)\subset X$ denote the set of points at which $\ph$ is discontinuous.  Then we let $\Af$ denote the class of Borel measurable functions $\ph\colon X\to \RR$ which satisfy the following conditions:
\begin{enumerate}[(A)]
\item $\ph$ is bounded (both above and below);
\item $\mu(\ccph) = 0$ for all $\mu\in \Mf(X)$.
\end{enumerate}
In particular, $\Af$ includes all continuous functions $\ph\in \CCC(X,\RR)$.  It also includes all bounded measurable functions $\ph$ for which $\CCC(\ph)$ is finite and contains no periodic points, and more generally, all bounded measurable functions for which $\ccph$ is disjoint from all its iterates.

We will see later (Proposition~\ref{prop:convergence}) that passing from $\CCC(X,\RR)$ to $\Af$ does not change the weak* topology at measures in $\Mf(X)$, which is the key to including these particular discontinuous functions in our results.

\begin{theorem}[The entropy spectrum for Birkhoff averages]\label{thm:birkhoff}
Let $X$ be a compact metric space, $f\colon X\to X$ be continuous, and $\ph\in \Af$.  Then
\begin{enumerate}[I.]
\item $\Tb$ is the Legendre transform of the Birkhoff spectrum:
\begin{equation}\label{eqn:TisBL}
\Tb(q) = \BL(q) = \sup_{\alpha\in\RR} (\Ba + q\alpha)
\end{equation}
for every $q\in \RR$.
\item The domain of $\Ba$ is bounded by the following:
\begin{align}
\label{eqn:amin}
\amin &= \inf \{ \alpha\in \RR \mid \Tb(q) \geq q\alpha \text{ for all } q \}, \\
\label{eqn:amax}
\amax &= \sup \{ \alpha\in \RR \mid \Tb(q) \geq q\alpha \text{ for all } q \},
\end{align}
That is, $\Kab = \emptyset$ for every $\alpha<\amin$ and every $\alpha>\amax$.
\item Suppose that $\Tb$ is $\CCC^r$ on $(q_1,q_2)$ for some $r\geq 1$, and that for each $q\in(q_1,q_2)$, there exists a (not necessarily unique) equilibrium state $\nu_q$ for the potential function $q\ph$.  Let $\alpha_1 = D^+\Tb(q_1)$ and $\alpha_2 = D^-\Tb(q_2)$; then 
\begin{equation}\label{eqn:BisTL}
\Ba=\TbL(\alpha) = \inf_{q\in\RR} (\Tb(q) - q\alpha)
\end{equation}
for all $\alpha\in (\alpha_1,\alpha_2)$.  In particular, $\Ba$ is strictly concave on $(\alpha_1,\alpha_2)$, and $\CCC^r$ except at points corresponding to intervals on which $\Tb$ is affine.
\end{enumerate}
\end{theorem}

Observe that the first two statements hold for \emph{every} continuous map $f$, without any assumptions on the system, thermodynamic or otherwise.  For discontinuous potentials in $\Af$, these are the first rigorous multifractal results of any sort known to the author.

Using the maps $A$ and $Q$ introduced above, Part III can be stated as follows:  if $\Tb$ is $\CCC^r$ on an open interval $I_Q$ and equilibrium states exist for all $q\in I_Q$, then~\eqref{eqn:BisTL} holds for all $\alpha\in A(I_Q)$.  If in addition $\Tb$ is strictly convex on $I_Q$, then $\Ba$ is $\CCC^r$ on $A(I_Q)$.

We will show later that if the entropy map is upper semi-continuous, then the conclusion of Part III holds at $\alpha_1$ and $\alpha_2$ as well.  We will also see (Proposition~\ref{prop:unique-works}) that existence of a \emph{unique} equilibrium state on an interval $(q_1,q_2)$ is enough to guarantee differentiability, and hence to apply Theorem~\ref{thm:birkhoff}.  As shown in Example~\ref{eg:vw} below, though, we may have differentiability without uniqueness.


\section{Phase transitions and generalisations of Theorem~\ref{thm:birkhoff}}\label{sec:gen}

\begin{figure}[tbp]
	\includegraphics{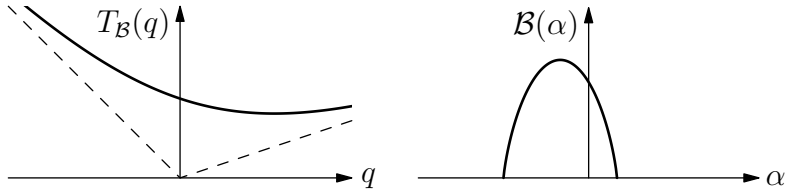}
	\caption{The Birkhoff spectrum for a map with no phase transitions.}
	\label{fig:no-phase-transition}
\end{figure}

If $\Tb$ is continuously differentiable for all $q$, then we obtain the complete Birkhoff spectrum, as shown in Figure~\ref{fig:no-phase-transition}.  However, there are many physically interesting systems which display \emph{phase transitions}---that is, values of $q$ at which $\Tb$ is non-differentiable.  For example, if $f\colon [0,1]\to [0,1]$ is the Manneville--Pomeau map and $\ph$ is the geometric potential $\log \abs{f'}$, then $\Tb$ is as shown in Figure~\ref{fig:phase-transition}~\cite{kN00}; in particular, $\Tb$ is not differentiable at $q_0$.  Thus Theorem~\ref{thm:birkhoff} gives the Birkhoff spectrum on the interval $[\alpha_1,\alpha_2]$, where $\alpha_1 = \lim_{q\to q_0^+} \Tb'(q)$, but says nothing about the interval $[0,\alpha_1)$, on which $\TbL(\alpha) = -q_0 \alpha$ is linear.

\begin{figure}[tbp]
	\includegraphics{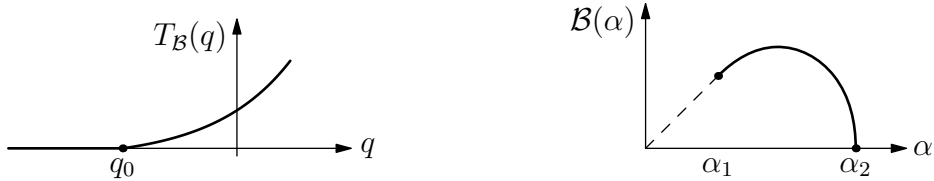}
	\caption{A phase transition in the Manneville--Pomeau map.}
	\label{fig:phase-transition}
\end{figure}

In fact, it is known that for this particular example, we have $\Ba = \TbL$ even on the linear stretch corresponding to the point of non-differentiability of $\Tb$~\cite{kN00}.  However, this is not universally the case, as may be seen by ``gluing together'' two unrelated maps.  Consider two maps $f_1\colon X_1 \to X_1$ and $f_2\colon X_2 \to X_2$, where $X_1$ and $X_2$ are disjoint, and suppose that the thermodynamic functions are as shown in Figure~\ref{fig:non-transitive}.  Let $X = X_1 \cup X_2$, and define a map $f\colon X\to X$ such that the restriction of $f$ to $X_i$ is $f_i$ for $i=1,2$.  Then $\Tb(q) = P^*(q\ph) = \max\{P_1^*(q\ph|_{X_1}), P_2^*(q\ph|_{X_2})\}$, where $P_i^*$ denotes the pressure of $f_i$, and furthermore $\Ba$ is the maximum of $\htop (\Kab\cap X_1)$ and $\htop (\Kab\cap X_2)$.  Thus $\Tb$ is non-differentiable at $q=0$, which corresponds to the interval $[\alpha_2,\alpha_3]$ on which $\TbL$ is constant.  Applying Theorem~\ref{thm:birkhoff} to each of the subsystems $f_i$, we see that $\Ba=\TbL(\alpha)$ on $[\alpha_1,\alpha_2]$ and $[\alpha_3,\alpha_4]$, but that the two are not equal on $(\alpha_2,\alpha_3)$, and that $\Ba$ is not concave on this interval.

\begin{figure}[tbp]
	\includegraphics{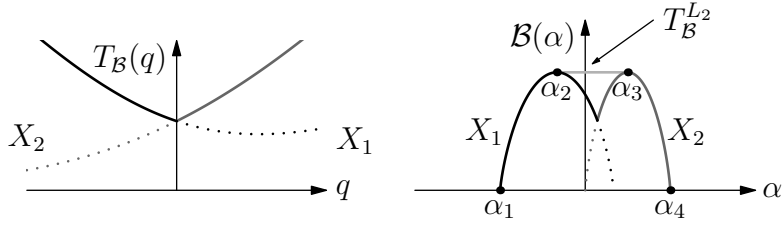}
	\caption{A different sort of phase transition.}
	\label{fig:non-transitive}
\end{figure}

\begin{example}\label{eg:vw}
Given $m,n\in \NN$, let $(X_1,f_1) = (\Sigma_m^+,\sigma)$ and $(X_2,f_2)=(\Sigma_n^+,\sigma)$ be the full one-sided shifts on $m$ and $n$ symbols, respectively, and construct $f\colon X\to X$ as above, where $X=X_1 \cup X_2$.  Choose two vectors $v\in \RR^m$ and $w\in \RR^n$, and let $\ph\colon X\to \RR$ be given by
\[
\ph(x) = \begin{cases}
v_{x_1} &x=x_1 x_2 \dots \in X_1 = \Sigma_m^+, \\
w_{x_1} &x=x_1 x_2 \dots \in X_2 = \Sigma_n^+.
\end{cases}
\]
Then an easy computation using the classical definition of pressure and the variational principle shows that
\begin{multline*}
\Tb(q) = P^*(q\ph) = \max(P_1^*(q\ph), P_2^*(q\ph)) \\
= \max\left(\log \left(\sum_{i=1}^m e^{qv_i}\right), \log \left(\sum_{j=1}^n e^{qw_j} \right) \right).
\end{multline*}
In particular, we see that $P_1^*(0) = \log m$ and $P_2^*(0) = \log n$, and also that
\begin{equation}\label{eqn:derivatives}
\begin{aligned}
\frac{d^k}{dq^k}P_1^*(q\ph)|_{q=0} &= \log\left(\sum_i v_i^k\right), \\
\frac{d^k}{dq^k}P_2^*(q\ph)|_{q=0} &= \log\left(\sum_j w_j^k\right).
\end{aligned}
\end{equation}
By judicious choices of $v$ and $w$, we can observe a variety of behaviours in the Birkhoff spectrum $\Ba$.  If $m=n$ but $\sum_i v_i \neq \sum_j w_j$, we obtain the picture shown in Figure~\ref{fig:non-transitive}.  

If $m=n$ and $\sum_i v_i = \sum_j w_j$, but $\sum_i v_i^2 > \sum_j w_j^2$, then the two pressure functions $P_1^*(q\ph)$ and $P_2^*(q\ph)$ are tangent at $q=0$, corresponding to the existence of two ergodic measures of maximal entropy (one on $X_1$ and one on $X_2$), but for values of $q$ near $0$, there is a unique equilibrium state supported on $X_1$.

Finally, if $m=n$ and $\sum_i v_i^k = \sum_j w_j^k$ for $k=1,2$, but not for $k=3$, then the two pressure functions are still tangent at $q=0$, but now the equilibrium state passes from $X_1$ to $X_2$ as $q$ passes through $0$.  Despite this transition and the non-uniqueness of the measure of maximal entropy, the pressure function $\Tb$ is still differentiable at $0$.
\end{example}

Having seen two very different manifestations of phase transitions (Figures~\ref{fig:phase-transition} and~\ref{fig:non-transitive}), we see that any generalisation of Theorem~\ref{thm:birkhoff} that treats phase transitions must somehow distinguish between these two sorts of behaviour.  The key difference is that in the first case, the system $f\colon X\to X$ can be approximated from within by a sequence of subsystems $X_n$ on which there is no phase transition---that is, the following condition holds~\cite{kN00,GR09}:
\begin{description}
\item[(A)]  There exists a sequence of compact $f$-invariant subsets $X_n\subset X$ such that the pressure function $q\mapsto P_{X_n}^*(q\ph)$ is continuously differentiable for all $q\in \RR$ (and equilibrium states exist), and furthermore,
\begin{equation}\label{eqn:pressures-converge}
\lim_{n\to\infty} P_{X_n}^*(q\ph) = P^*(q\ph).
\end{equation}
\end{description}
This condition fails for the example in Figure~\ref{fig:non-transitive}, in which the phase transition represents a jump from one half of the system to the other half, which is disconnected from the first, rather than an escaping of measures to an adjacent fixed point.  Using Condition \Aa, we can state a general theorem which extends Theorem~\ref{thm:birkhoff} to maps for which $\Tb$ has points of non-differentiability.

\begin{theorem}\label{thm:phase}
Let $X$ be a compact metric space, $f\colon X\to X$ be continuous, and $\ph\in \Af$.  If Condition \Aa\ holds, then we have~\eqref{eqn:BisTL} for all $\alpha\in (\amin,\amax)$.
\end{theorem}

As mentioned just before Theorem~\ref{thm:birkhoff}, the key property of potentials $\ph \in \Af$ is that weak* convergence to an invariant measure implies convergence of the integrals of $\ph$; this is the only place in the proof where we use the requirement that $\ph$ lie in $\Af$.  

For potentials outside of $\Af$, we can try to regain approximate convergence results at certain relevant measures by using the topological entropy of $\ccph$ to give a bound on how much weight a neighbourhood of $\ccph$ carries.

To this end, given $h\geq 0$, consider the set
\[
I_A(h) = \{\alpha\in \RR \mid \TbL(\alpha) > h\},
\]
and also its counterpart
\[
I_Q(h) = Q(I_A(h)).
\]
Geometrically, $I_Q(h)$ may be described as the set of values $q\in \RR$ such that there is a line through $(q,\Tb(q))$ that lies on or beneath the graph of $\Tb$ and intersects the $y$-axis somewhere above $(0,h)$.

\begin{theorem}\label{thm:high-entropy}
Let $X$ be a compact metric space, $f\colon X\to X$ be continuous, and $\ph\colon X\to \RR$ be measurable and bounded (above and below).  Let $\CCC(\ph)$ be the set of discontinuities of $\ph$, and let $h_0 = \lhtop(\CCC(\ph))$.  Then 
\begin{enumerate}[I.]
\item For every $q\in I_Q(h_0)$, we have the following version of~\eqref{eqn:TisBL}:
\begin{equation}\label{eqn:TisBL2}
\Tb(q) = \sup_{\alpha\in I_A(h_0)} (\Ba + q\alpha).
\end{equation}
\item $\Ba \leq h_0$ for every $\alpha\notin I_A(h_0)$.
\item Suppose that $\Tb$ is $\CCC^r$ on $(q_1,q_2)\subset Q(h_0)$ for some $r\geq 1$, and that for each $q\in (q_1,q_2)$ there exists a (not necessarily unique) equilibrium state $\nu_q$ for the potential function $q\ph$.  Then~\eqref{eqn:BisTL} holds for all $\alpha\in (\alpha_1,\alpha_2) = A((q_1,q_2))$.
\end{enumerate}
\end{theorem}

Finally, although we are not yet able to give a complete treatment of unbounded potential functions, we can show that everything works if our potential function is bounded below and we only consider $q\leq 0$.

\begin{theorem}\label{thm:singularity}
Let $X$ be a compact metric space, $f\colon X\to X$ be continuous, and $\ph\colon X\to \RR \cup \{+\infty\}$ be continuous where finite (and hence bounded below).  Let $\alpha_0 D^-\Tb(0)$, and let $\amin$ be given by~\eqref{eqn:amin}, so $(\amin,\alpha_0) = A((-\infty,0))$.  Then 
\begin{enumerate}[I.]
\item For every $q\leq 0$,~\eqref{eqn:TisBL} holds.
\item For $\alpha<\amin$, we have $\Kab = \emptyset$.
\item Suppose that $\Tb$ is $\CCC^r$ on $(q_1,q_2)$ for some $r\geq 1$ and $q_1 < q_2\leq 0$, and that for each $q\in (q_1,q_2)$ there exists a (not necessarily unique) equilibrium state $\nu_q$ for the potential $q\ph$.  Then~\eqref{eqn:BisTL} holds for all $\alpha\in (\alpha_1,\alpha_2) = A((q_1,q_2))$.
\end{enumerate}
\end{theorem}

An analogous result holds for $q\geq 0$ if $\ph$ is bounded above but not below.  Also, as with Theorem~\ref{thm:birkhoff}, Part III extends to the endpoints $\alpha_i$ if the entropy map is upper semi-continuous.

\section{Conformal maps and Lyapunov spectra}\label{sec:conformal}

\begin{definition}
We say that a continuous map $f\colon X\to X$ is \emph{conformal} with factor $a(x)$ if for every $x\in X$ we have
\begin{equation}\label{eqn:conformal}
a(x) = \lim_{y\to x} \frac{d(f(x),f(y))}{d(x,y)},
\end{equation}
where $a\colon X\to [0,\infty)$ is continuous.  A point $x\in X$ is a \emph{critical point} of $f$ if $a(x)=0$.  We denote the Birkhoff sums of $\log a$ by
\[
\lambda_n(x) = \frac 1n S_n (\log a)(x),
\]
and consider the lower and upper limits
\[
\llambda(x) = \llim_{n\to\infty} \lambda_n(x), \qquad
\ulambda(x) = \ulim_{n\to\infty} \lambda_n(x).
\]
If they agree (that is, if the limit exists), we write
\[
\lambda(x) = \lim_{n\to\infty} \lambda_n(x)
\]
for the \emph{Lyapunov exponent} at $x$.  Given a measure $\mu\in\MMM(X)$ we define the Lyapunov exponent of $\mu$ as
\[
\lambda(\mu) = \int_X \lambda(x) \,d\mu(x).
\]
If $\mu$ is ergodic, then $\lambda(\mu)=\lambda(x)$ for $\mu$-almost every $x\in X$.
\end{definition}

Note that in the case where $X$ is a smooth Riemannian manifold, the definition of conformality may be restated as the requirement that $Df(x)$ is $a(x)$ times some isometry, and the definition of Lyapunov exponent becomes the usual one from smooth ergodic theory.  In particular, if $X$ is one-dimensional, then any differentiable map is conformal.

Denote by $\BB$ the set of all points in $X$ which satisfy (at least) one of the following two conditions.
\begin{description}
\item[(B1)]  \emph{Bounded contraction:}  $\inf \{ S_n(\log a)(f^k(x)) \mid k,n\in\NN \} > -\infty$.  Note that if $a(x) \geq 1$ for all $x\in X$, then $f$ has no contraction whatsoever (although the expansion may not be uniform), and so every point has bounded contraction.
\item[(B2)]  \emph{Lyapunov exponent exists:}  $\llambda(x)=\ulambda(x)$.
\end{description}

The following lemma is proved in~\cite{vC09a}, and shows that we can dynamically generate metric balls using conformal maps without critical points.  We will need this later for the results on $\Da$ in Section~\ref{sec:wkgibbs}.  

\begin{lemma}\label{lem:well-behaved}
Let $X$ be a compact metric space and $f\colon X\to X$ be continuous and conformal with factor $a(x)$.  Suppose that $f$ has no critical points; that is, that $a(x)>0$ for all $x\in X$.  Then given any $x\in \BB$ and $\eps>0$, there exists $\delta=\delta(\eps)>0$ and $\eta=\eta(x)>0$ such that for every $n$,
\begin{equation}\label{eqn:diamball}
B\left(x,\eta\delta e^{-n(\lambda_n(x) + \eps)}\right) \subset B(x,n,\delta) \subset 
B\left(x,\delta e^{-n(\lambda_n(x) -\eps)}\right).
\end{equation}
\end{lemma}

Using this result, it is shown in~\cite{vC09a} that if $f$ is a conformal map without critical points, then given $Z\subset X$ and $\alpha>0$ such that
\begin{equation}\label{eqn:lyapconst}
\llambda(x) = \ulambda(x) = \alpha
\end{equation}
for every $x\in Z$, the Hausdorff dimension and topological entropy of $Z$ are related by
\begin{equation}\label{eqn:htopdimh}
\dim_H Z = \frac 1\alpha \htop Z.
\end{equation}

Recall that the level sets $\Kab$ for the Birkhoff averages of the geometric potential $\ph = \log a$ are precisely the level sets $\Kal$ for the Lyapunov exponents of $f$, and thus $\LEa = \Ba$ is determined by $\Tb$ using Theorem~\ref{thm:birkhoff}.  Since every point $x\in \Kal$ satisfies~\eqref{eqn:lyapconst}, we may apply~\eqref{eqn:htopdimh} and obtain
\begin{equation}\label{eqn:dimlyap}
\LDa = \frac 1\alpha \LEa
\end{equation}
for all $\alpha>0$.  Thus both Lyapunov spectra can be determined in terms of the Legendre transform of $\Tb$, provided equilibrium states exist and $\Tb$ is differentiable.  We stress that since $\LDa$ is not given by a Legendre transform, but is obtained by a rescaling, it may not be convex---see~\cite{IK09} for examples where this occurs.

\section{Entropy and dimension spectra of weak Gibbs measures}\label{sec:wkgibbs}

The two remaining multifractal spectra with which we are concerned---the entropy spectrum and the dimension spectrum---are both defined in terms of a measure $\mu$.  In order to relate these spectra to the thermodynamic quantities associated with a potential $\ph$, we need a relationship between the local properties of $\mu$ and the Birkhoff averages of $\ph$.  This is provided by the notion of a weak Gibbs measure.

\begin{definition}
Given a compact metric space $X$, a continuous map $f\colon X\to X$, and a potential $\ph\colon X\to \RR$ (not necessarily continuous), we say that a Borel probability measure $\mu$ is a \emph{weak Gibbs measure} for $\ph$ with constant $P\in \RR$ if for every $x\in X$ and $\delta>0$ there exists a sequence $M_n = M_n(x,\delta) > 0$ such that
\begin{equation}\label{eqn:Gibbs}
\frac 1{M_n} \leq \frac{\mu(B(x,n,\delta))}{\exp(-nP + S_n\ph(x))} \leq M_n
\end{equation}
for every $n\in \NN$, where we require the following growth condition on $M_n$ to hold for every $x\in X$:
\begin{equation}\label{eqn:tempered}
\lim_{\delta\to 0} \ulim_{n\to\infty} \frac 1n \log M_n(x,\delta) = 0.
\end{equation}
\end{definition}

There are various definitions in the literature of Gibbs measures of one sort or another; most of these definitions agree in spirit, but differ in some slight details.  We note the differences between the above definition and other definitions in use.
\begin{enumerate}[(1)]
\item The classical definition (see~\cite{rB75}) requires $M_n$ to be bounded, not just to have slow growth, as we require here.  In that case the sequence $M_n$ can be (and is) replaced by a single constant $M$.  The notion of a weak Gibbs measure, for which the constant can vary slowly in $n$, is used in~\cite{mY00,mK01,FO03,JR09}, among others.
\item The above definitions all require the constant $M$ to be independent of $x$, whereas we require no such uniformity.  Furthermore, they are given in terms of cylinder sets rather than Bowen balls; we follow~\cite{VV08} in using the latter, as this is what we need for the multifractal analysis.
\item Certain authors only require~\eqref{eqn:Gibbs} to hold for $\mu$-a.e.\ $x\in X$~\cite{mY00,VV08}.  In order to do the multifractal analysis, we need conditions which hold everywhere, not just almost everywhere, and so we require~\eqref{eqn:Gibbs} for \emph{every} point $x\in X$.
\item Following Kesseb\"ohmer~\cite{mK01}, we do not \emph{a priori} require that a weak Gibbs measure be $f$-invariant.  Weak Gibbs measures exist for \emph{any} continuous function $\ph$ on a one-sided shift space~\cite{mK01}, but it is not the case that such measures can always be taken to be invariant.
\end{enumerate}

We have given the definition in the above form because~\eqref{eqn:Gibbs} is reminiscent of the usual definition of Gibbs measure.  For our purposes, an alternate form of~\eqref{eqn:Gibbs} will be more useful:
\begin{equation}\label{eqn:Gibbs2}
\left\lvert -\frac 1n \log \mu(B(x,n,\delta)) + \frac 1n S_n \ph(x) - P \right\rvert \leq \frac 1n\log M_n(x,\delta) \to 0,
\end{equation}
where the limit is taken as $n\to \infty$ and then as $\delta \to 0$.  Given an invariant weak Gibbs measure, it follows from~\eqref{eqn:Gibbs2} that $h_\mu(x)$ exists if and only if $\ph^+(x)$ exists, and that in this case
\begin{equation}\label{eqn:handph}
h_\mu(x) + \ph^+(x) = P.
\end{equation}

If $\ph$ is continuous, then dimensional arguments from~\cite{yP98} show that $P$ is equal to the topological pressure $P_X(\ph)$, and thus the variational principle shows that it is equal to $P^*(\ph)$.  Integrating~\eqref{eqn:handph} with respect to $\mu$, we obtain $P^*(\ph) = h(\mu) + \int \ph \,d\mu$, hence $\mu$ is an equilibrium state.  Thus a weak Gibbs measure is an equilibrium state, just as in the classical case.

For any equilibrium state, the Brin--Katok entropy formula and the Birkhoff ergodic theorem together imply that~\eqref{eqn:handph} holds almost everywhere with $P=P^*(\ph)$; our definition of weak Gibbs measure boils down to requiring that it hold \emph{everywhere}, without placing any extra requirements on uniformity or rate of convergence.

Writing $\ph_1(x) = \ph(x) - P^*(\ph)$, we observe that
\begin{equation}\label{eqn:levelsets}
\Kab(\ph_1) = K_{-\alpha}^\EEE,
\end{equation}
and we may thus obtain $\Ea$ as a Legendre transform of the following function:
\[
\Te(q) = P^*(q\ph_1).
\]
Once again, convexity of $\Te$ is immediate from the definition of $P^*$.  The following theorem is a direct consequence of Theorem~\ref{thm:birkhoff} and~\eqref{eqn:handph}; because of the change of sign in~\eqref{eqn:levelsets}, we must use the following versions of the Legendre transform:
\begin{equation}\label{eqn:TLSL}
\begin{aligned}
\TTL(\alpha) &= \inf_{q\in\RR} (T(q) + q\alpha), \\
\SSL(q) &= \sup_{\alpha\in\RR} (S(\alpha) - q\alpha).
\end{aligned}
\end{equation}
Note that there is a corresponding change of sign in the definitions of the maps $A$ and $Q$.

\begin{theorem}[The entropy spectrum for local entropies]\label{thm:entropy}
Let $X$ be a compact metric space, $f\colon X\to X$ be continuous, and $\ph\in \Af$.  Then if $\mu$ is a weak Gibbs measure for $\ph$, we have the following:
\begin{enumerate}[I.]
\item $\Te$ is the Legendre transform of the entropy spectrum:
\begin{equation}\label{eqn:TisEL}
\Te(q) = \EL(q) = \sup_{\alpha\in\RR} (\Ea - q\alpha)
\end{equation}
for every $q\in \RR$.
\item The domain of $E$ is bounded by the following:
\begin{align*}
\amin &= \inf \{ \alpha\in \RR \mid \Te(q) \geq -q\alpha \text{ for all } q \}, \\
\amax &= \sup \{ \alpha\in \RR \mid \Te(q) \geq -q\alpha \text{ for all } q \},
\end{align*}
That is, $\Kae = \emptyset$ for every $\alpha<\amin$ and every $\alpha>\amax$.
\item Suppose that $\Te$ is $\CCC^r$ on $(q_1,q_2)$ for some $r\geq 1$, and that for each $q\in(q_1,q_2)$, there exists a (not necessarily unique) equilibrium state $\nu_q$ for the potential function $q\ph_1$.  Let $\alpha_1 = -D^+\Te(q_1)$ and $\alpha_2 = D^-\Te(q_2)$.  Then 
\begin{equation}\label{eqn:EisTL}
\Ea=\TeL(\alpha) = \inf_{q\in\RR} (\Te(q) + q\alpha)
\end{equation}
for all $\alpha\in (\alpha_2,\alpha_1)$; in particular, $E$ is strictly concave on $(\alpha_2,\alpha_1)$, and $\CCC^r$ except at points corresponding to intervals on which $\Te$ is affine.
\end{enumerate}
\end{theorem}

In the case where $f$ is conformal, we prove the analogous result for the dimension spectrum.  We will need to eliminate points at which the Birkhoff averages of $\log a$ cluster around zero along a sequence of times at which the local entropy of $\mu$ is also negligible; that is, the following set:
\begin{equation}\label{eqn:Z}
\ZZZ(\mu) = \left\{ x\in X \,\Big|\, \lim_{\delta\to 0} \llim_{n\to\infty} 
\left\lvert \frac 1n \log\mu(B(x,n,\delta)) \right\rvert + \left\lvert \frac 1n S_n \log a(x) \right\rvert = 0 
\right\}.
\end{equation}
When $\mu$ is a weak Gibbs measure for $\ph$, we have
\begin{equation}\label{eqn:Z2}
\ZZZ(\mu) = \left\{ x\in X \,\Big|\, \llim_{n\to\infty} 
\left\lvert \frac 1n S_n \ph_1(x) \right\rvert + \left\lvert \frac 1n S_n \log a(x) \right\rvert = 0 
\right\}.
\end{equation}
In the context of Theorem~\ref{thm:dimension}, we will suppress the dependence on $\mu$ and simply write $\ZZZ = \ZZZ(\mu)$.  We will see that the set $\ZZZ$ contains all points $x$ for which $\llambda(x) = 0$ but $\ud_\mu(x)<\infty$; these are the only points our methods cannot deal with.  In many cases we do not lose much by neglecting them; for example, if $\sup \ph - \inf \ph < h(\mu)$, then
\[
\ulim_{n\to\infty} \frac 1n S_n \ph_1(x) < 0
\]
for every $x\in X$, and so $\ZZZ = \emptyset$.  Even in cases when $\ZZZ$ is non-empty, it often has zero Hausdorff dimension~\cite{JR09}.

The remaining set of ``good'' points will be denoted by
\begin{equation}\label{eqn:X'}
X' = X \setminus \ZZZ.
\end{equation}
In the definition of $\Da$, we adopt the convention that $\Da = -\infty$ if $\Kad \subset \ZZZ$.  Since there may be points at which $\mu$ has infinite pointwise dimension, we also include the value $\alpha=+\infty$ in~\eqref{eqn:TLSL}, and follow the convention that if $K_\infty^\DDD \cap X' \neq \emptyset$, then $\DL(q) = +\infty$ for all $q<0$.

Now consider the centred potential $\ph_1(x) = \ph(x) - P^*(\ph)$.  Define a family of potentials by
\begin{equation}\label{eqn:conformal2}
\ph_{q,t}(x) = q\ph_1(x) - t\log a(x).
\end{equation}
We will be particularly interested in the potentials with zero pressure; we would like to define a function $\Td(q)$ by the equation
\begin{equation}\label{eqn:Pzero}
P^*\left(\ph_{q,\Td(q)}\right) = 0.
\end{equation}
Formally, we write
\begin{equation}\label{eqn:Td}
\Td(q) = \inf \{ t\in\RR \mid P^*(\ph_{q,t}) \leq 0 \} = \sup\{t\in \RR \mid P^*(\ph_{q,t}) > 0\};
\end{equation}
by continuity of $P^*$, $\Td(q)$ solves~\eqref{eqn:Pzero} if it is finite, but is not necessarily the unique solution of~\eqref{eqn:Pzero}.  (Indeed, there may be values of $q$ for which $P^*(\ph_{q,t}) = 0$ for all $t>\Td(q)$.)

For $\Td(q)<\infty$ we write $\ph_q = \ph_{q,\Td(q)}$, and observe that~\eqref{eqn:Pzero} may be written as $P^*(\ph_q)=0$.

Given $\eta>0$ and $I_Q\subset \RR$, we will need to consider the following region lying just under the graph of $\Td(q)$:
\[
R_\eta(I_Q) = \{(q,t)\in \RR^2 \mid q\in I_Q, \Td(q) - \eta < t < \Td(q) \}.
\]

We can now state a general result regarding the dimension spectrum.

\begin{theorem}[The dimension spectrum for pointwise dimensions]\label{thm:dimension}
Let $X$ be a compact metric space with $\dim_H X < \infty$, and let $f\colon X\to X$ be continuous and conformal with continuous non-vanishing factor $a(x)$.  Suppose that $\BB=X$ and that $\lambda(\nu)\geq 0$ for every $\nu\in \Mf(X)$.  Let $\mu\in \Mf(X)$ be a weak Gibbs measure for a continuous potential $\ph$.  Finally, suppose that $\dim_H \ZZZ = 0$.  Then we have the following.
\begin{enumerate}[I.]
\item $\Td$ is the Legendre transform of the dimension spectrum:
\begin{equation}\label{eqn:TisDL}
\Td(q) = \DL(q) = \sup_{\alpha\in\RR} (\Da - q\alpha)
\end{equation}
for every $q\in \RR$.
\item Neglecting points in $\ZZZ$, the domain of $\DDD$ is bounded by the following:
\begin{align*}
\amin &= \inf \{ \alpha\in \RR \mid \Td(q) \geq -q\alpha \text{ for all } q \}, \\
\amax &= \sup \{ \alpha\in \RR \mid \Td(q) \geq -q\alpha \text{ for all } q \},
\end{align*}
That is, $\Kad \cap X' = \emptyset$ for every $\alpha<\amin$ and every $\alpha>\amax$.
\item Suppose $I_Q = (q_1,q_2)$ and $\eta>0$ are such that for every $(q,t)\in R_\eta(I_Q)$, the potential $\ph_{q,t}$ has a (not necessarily unique) equilibrium state, and that the map $(q,t) \mapsto P^*(\ph_{q,t})$ is $\CCC^r$ on $R_\eta(I_Q)$ for some $r\geq 1$.  Then we have
\begin{equation}\label{eqn:DisTL}
\Da = \TdL(\alpha) = \inf_{q\in\RR} (\Td(q) + q\alpha)
\end{equation}
for all $\alpha\in (\alpha_2,\alpha_1) = A(I_Q)$; in particular, $\DDD$ is strictly concave on $(\alpha_2,\alpha_1)$, and $\CCC^r$ except at points corresponding to intervals on which $\Td$ is affine.
\end{enumerate}
\end{theorem}

We will see in the proof that the requirement on existence of equilibrium states for $\ph_{q,t}$ with $(q,t)\in R_\eta(I_Q)$ can be replaced by the condition that there exist equilibrium states $\nu_q$ for $\ph_q = \ph_{q,\Td(q)}$ such that $\lambda(\nu_q) > 0$.  However, such measures do not necessarily exist, while upper semi-continuity of the entropy is enough to guarantee the existence of the measures required in the theorem.

If we do have equilibrium states $\nu_q$ with $\lambda(\nu_q) > 0$, then in Part III of the theorem, the requirement that $(q,t)\mapsto P^*(\ph_{q,t})$ be $\CCC^r$ on $R_\eta(I_Q)$ can be replaced by the condition that $\Td$ be $\CCC^r$ on $I_Q$.

\section{Remarks}\label{sec:rmk}

We first discuss conditions under which the hypotheses of Theorem~\ref{thm:birkhoff} and the results in Section~\ref{sec:gen} are satisfied, before turning our attention to weak Gibbs measures and Theorem~\ref{thm:entropy}, and finally the more delicate case of Theorem~\ref{thm:dimension}.  Throughout this section, $T$ will refer to any or all of $\Tb$, $\Te$, and $\Td$, as needed.  We make general remarks in this section, deferring specific examples and applications until Section~\ref{sec:app}.

\subsection{Birkhoff spectrum---continuous potentials}

Parts I and II of Theorem~\ref{thm:birkhoff} and~\ref{thm:entropy} do not place any thermodynamic requirements on the function $T=\Tb$, and thus hold in full generality.

There are two thermodynamic requirements in Part III---existence of an equilibrium state, and differentiability of $T$.  The latter is used in order to guarantee the existence of values $q\in \RR$ for which $\Tb'(q)$ exists, and hence $A(q) = \{\Tb'(q)\}$ is a singleton.  In fact, because $T$ is continuous and convex, $A(q)$ is a singleton for all but at most countably many values of $q$, and consequently, once existence of equilibrium states is established, it follows that the Birkhoff spectrum is equal to the Legendre transform of the pressure function everywhere except possibly on some countable union of intervals, on each of which that Legendre transform is affine and gives an upper bound for $\Ba$.

Existence of equilibrium states is easy to verify in the following rather common setting.

\begin{definition}
The entropy map $\mu\mapsto h(\mu)$ is \emph{upper semi-continuous} if for every sequence $\mu_n\in\MMM(X)$ which converges to $\mu$ in the weak* topology, we have
\[
\ulim_{n\to\infty} h(\mu_n) \leq h(\mu).
\]
\end{definition}

If the entropy map is upper semi-continuous and $\ph$ is continuous, then the map
\[
\mu \mapsto h(\mu) + \int q\ph \,d\mu
\]
is upper semi-continuous for every $q\in\RR$, and thus attains its maximum.  In particular, there exists an equilibrium state for every $q\ph$.

\begin{definition}
$f$ is \emph{expansive} if there exists $\eps>0$ such that for all $x\neq y$ there exists $n\in \ZZ$ (if $f$ is invertible) or $n\in \NN$ (if $f$ is non-invertible) such that $d(f^n(x),f^n(y))\geq \eps$.
\end{definition}

For expansive homeomorphisms, the entropy map $\mu\mapsto h_\mu(f)$ is upper semi-continuous~\cite[Theorem 8.2]{pW75},\foot{What happens if $f$ is non-invertible?} and so existence is guaranteed for continuous $\ph$.  Similarly, the entropy map is upper semi-continuous for $\CCC^\infty$ maps of compact smooth manifolds~\cite{sN89}, and we once again get existence for free.

\begin{proposition}\label{prop:unique-works}
Let $X$ be a compact metric space, $f\colon X\to X$ a continuous map, and $\ph\in \Af$.  Suppose that the entropy map is upper semi-continuous and that there exists an interval $(q_1,q_2)\subset \RR$ such that for every $q\in (q_1,q_2)$, the potential $q\ph$ has a unique equilibrium state.  Then $\Tb$ is $\CCC^1$ on $(q_1,q_2)$.
\end{proposition}
\begin{proof}
Suppose for a contradiction that the pressure function $q\mapsto P^*(q\ph)$ is not differentiable at $q_0\in (q_1,q_2)$.  Let $\mu_n^-$ be the unique equilibrium state for $(q-\frac 1n)\ph$, and let $\mu^-$ be a weak* limit of some subsequence $\mu_{n_j}^-$.  By upper semi-continuity and Proposition~\ref{prop:convergence} below, we have
\begin{multline*}
h(\mu^-) + \int q\ph\,d\mu^- \geq \ulim_{n_j\to\infty} h(\mu_{n_j})+ \int q\ph\,d\mu_{n_j}^- \\
= \ulim_{n_j\to\infty} P^*\left(\left(q-\frac 1{n_j}\right)\ph\right) = P^*(q\ph).
\end{multline*}
Thus $\mu^-$ is an equilibrium state for $q\ph$ with
\[
\int q\ph\,d\mu^- = D^-\Tb(q) = \lim_{q'\to q^-} \Tb'(q')
\]
by Proposition~\ref{prop:ruelle} below.  Similarly, one can construct an equilibrium state $\mu^+$ such that $\int q\ph\,d\mu^+$ is the right derivative of $\Tb$ at $q$.  If the two derivatives do not agree, then we have two distinct equilibrium states for $q\ph$, a contradiction.
\end{proof}

Using Proposition~\ref{prop:unique-works}, one approach to verifying the hypotheses of Theorem~\ref{thm:birkhoff} for a map with upper semi-continuous entropy is to show that the equilibrium state for each $q\ph$ is unique.

We also observe that in the context of Part III of Theorem~\ref{thm:birkhoff}, the construction in the proof above gives equilibrium states for $q_1\ph$ and $q_2\ph$ that are supported on the sets $K_{\alpha_1}^\BBB$ and $K_{\alpha_2}^\BBB$, respectively, and which establish~\eqref{eqn:BisTL} for the endpoints $\alpha_1$ and $\alpha_2$, just as in the proof of Proposition~\ref{prop:concave} below.


\subsection{Birkhoff spectrum---discontinuous potentials}

If $\ph$ is discontinuous, the map from $\MMM(X)$ to $\RR$ defined by
\begin{equation}\label{eqn:intph}
\mu\mapsto \int \ph\,d\mu
\end{equation}
is not continuous on all of $\MMM(X)$.  For discontinuous potentials lying in $\Af$, continuity still holds at measures in $\Mf(X)$ by Proposition~\ref{prop:convergence} below, which suffices for all the proofs here.

However, if $\ph\notin \Af$, then there may be invariant measures at which the map is discontinuous.  In particular, if $\mu(\CCC(\ph)) > 0$, then the map in~\eqref{eqn:intph} is discontinuous at $\mu$.  If $\ph$ is unbounded, then it is relatively straightforward to show that the map is not continuous at \emph{any} measure in $\MMM(X)$.  In many cases, it is not even enough to restrict our attention to invariant measures~\cite[Proposition 2.8]{BK98}.  Thus for $\ph\notin \Af$, upper semi-continuity of the entropy is not enough to guarantee existence of equilibrium states without further information.


For potentials which are bounded above but not below, we observe in Proposition~\ref{prop:singularity} that the map in~\eqref{eqn:intph} is upper semi-continuous, and thus the free energy function $\mu\mapsto h(\mu) + \int \ph\,d\mu$ is upper semi-continuous as well.  It follows that it attains its maximum, and we once again are guaranteed existence.  This is also enough to prove Proposition~\ref{prop:unique-works} for these potentials, showing that existence and uniqueness imply differentiability of the pressure function (for the appropriate sign of $q$) if the entropy map is upper semi-continuous.


\subsection{Entropy spectrum---weak Gibbs measures}

There are many cases in which equilibrium states are known to have the weak Gibbs property~\eqref{eqn:Gibbs} or one which implies it.  For example, equilibrium states for H\"older continuous potentials on uniformly hyperbolic systems are known to be Gibbs, as are equilibrium states for potentials satisfying a certain regularity property on expansive maps with specification~\cite{TV99}.  Finally, Kesseb\"ohmer proves the existence of weak Gibbs measures for continuous potentials on symbolic space~\cite{mK01} (these measures are studied by Jordan and Rams~\cite{JR09} on parabolic interval maps).

Given a weak Gibbs measure, all the above remarks regarding the Birkhoff spectrum apply to the entropy spectrum.

\subsection{Dimension spectrum}

Because of the geometric implications of any result regarding the dimension spectrum, we must deal with a more restricted class of systems.  In particular, the present approach is completely dependent upon conformality of the map $f$; without conformality, we have no analogue of Lemma~\ref{lem:well-behaved} or Proposition~\ref{prop:localdim}.  If analogues of these can be found in the non-conformal case, then it may be possible to establish a non-conformal version of the present result; however, this appears to require the use of a non-additive version of the thermodynamic formalism~\cite{lB96,FH10}.

We also presently lack the tools to deal with maps with critical points.  To establish an analogue of Lemma~\ref{lem:well-behaved} for such maps would require an estimate on the rate of recurrence of fairly arbitrary orbits to the critical point in order to control the distortion.

The other hypotheses in Theorem~\ref{thm:dimension} are less restrictive, and are satisfied for quite general classes of maps.  We discuss them briefly.

$\BBB = X$.  If $a(x)\geq 1$ for all $x$, then this is automatically satisfied; we do not need $a(x)>1$, nor any uniformity, and so the class of systems with this property includes Manneville--Pomeau maps and parabolic rational maps.  Due to recurrence of the critical point, bounded contraction \emph{per se} cannot be expected to hold for maps with critical points; however, the requirement of bounded contraction can in fact be weakened slightly to include cases where the absolute value of the quantity in \textbf{(B1)} is not bounded, but grows sublinearly in $n+k$, which corresponds to a certain sort of slow recurrence.  This approach, however, has yet to bear fruit.

$\dim_H \ZZZ = 0$.  Points at which $\lambda(x) = 0$ and $d_\mu(x)<\infty$ are problematic for various reasons, and so we want to avoid having to deal with them.  Since all such points lie in the set $\ZZZ$, we can do this by neglecting $\ZZZ$ in all our computations, and it turns out that this is not a very heavy price to pay.  Of course if $f$ is uniformly expanding, this set is empty, but even in the non-uniformly expanding case, it is shown in~\cite{JR09} that $\ZZZ$ has zero Hausdorff dimension for a class of parabolic interval maps.

We also observe that if the entropy map is upper semi-continuous, then existence of equilibrium states for $\ph_{q,t}$ is guaranteed for all $q,t\in \RR$, and that uniqueness is again enough to establish differentiability of the map $(q,t)\mapsto P^*(\ph_{q,t})$, and hence to apply Theorem~\ref{thm:dimension}.

\section{Applications}\label{sec:app}

Before proceeding to the proofs of the theorems, we give several concrete applications.

\subsection{Birkhoff spectrum}

The first two parts of Theorem~\ref{thm:birkhoff} do not require \emph{any} hypotheses on the map $f$ beyond continuity, and so for every continuous map $f$ and every potential $\ph\in \Af$, the pressure function $\Tb$ is the Legendre transform of $\Ba$ (and hence $\TbL$ is the concave hull of $\Ba$), and the domain of the Birkhoff spectrum is the interval $[\amin,\amax]$.

Similar but weaker statements hold for arbitrary bounded measurable potentials $\ph$, using Theorem~\ref{thm:high-entropy}, and for potentials with singularities using Theorem~\ref{thm:singularity}.

To apply the full strength of these three theorems beyond the general remarks made so far, we need some thermodynamic information about the system.

\subsubsection{Uniform hyperbolicity}

In~\cite{rB75}, Bowen showed that if $M$ is a $\CCC^\infty$ Riemannian manifold and $f\colon M\to M$ is an Axiom A diffeomorphism, then any H\"older continuous potential function $\ph\colon M\to\RR$ has a unique equilibrium state.  Since such maps are expansive on the hyperbolic set~\cite[Corollary 6.4.10]{KH95}, this suffices to check the hypotheses of Theorem~\ref{thm:birkhoff}, as shown in the previous section, and hence the Birkhoff spectrum is equal to the Legendre transform of the pressure function:  in particular, it is concave and $\CCC^1$ (see Figure~\ref{fig:no-phase-transition}).  Versions of this result may be extracted from the results in~\cite{TV99,PW01}, but Theorem~\ref{thm:birkhoff} provides a more direct proof.

Non-H\"older potentials were studied by Pesin and Zhang in~\cite{PZ06} (see also~\cite{hH08}).  They consider a uniformly piecewise expanding full-branched Markov map $f$ of the unit interval, and use inducing schemes and tools from the theory of countable Markov shifts to study the existence and uniqueness of equilibrium states for a large class of potentials.  In particular, they give the following example of a non-H\"older potential:
\begin{equation}\label{eqn:nonHolder}
\ph(x) = \begin{cases}
-(1-\log x)^{-\alpha} & x\in (0,1], \\ 0 & x=0. \end{cases}
\end{equation}
It is shown in~\cite{PZ06} that for any $\alpha>1$ and $q\in \RR$, the potential $q\ph$ has a unique equilibrium state.  Since $f$ is expansive, by the comments in the previous section this suffices to check the hypotheses of Theorem~\ref{thm:birkhoff}, and we have the following result.

\begin{proposition}\label{prop:nonHolder}
Let $f$ be a uniformly piecewise expanding full-branched Markov map of the unit interval, and let $\ph$ be the potential function given in~\eqref{eqn:nonHolder}, $\alpha>1$.  Then the Birkhoff spectrum $\Ba$ is smooth and concave, has domain $[\amin,\amax]$, and is the Legendre transform of $\Tb$.
\end{proposition}

Indeed, Proposition~\ref{prop:nonHolder} also holds for any potential $\ph$ such that all $q\ph$ are in the class considered by Pesin and Zhang.

For $0<\alpha\leq 1$, it is shown in~\cite{PZ06} that $\Tb$ has a phase transition at some value $q_0>0$.  Applying Theorem~\ref{thm:birkhoff}, we obtain a result for the non-linear part of the Birkhoff spectrum (see Figure~\ref{fig:phase-transition}); to obtain a complete result, we would need to apply Theorem~\ref{thm:phase} by establishing Condition \Aa.  Although this remains open, one might attempt to do this by using the fact that for a potential with summable variations, the Gurevich pressure on a topologically mixing countable Markov shift $X$ is the supremum of the classical topological pressure over topologically mixing finite Markov subshifts of $X$~\cite{oS99}; these finite subshifts give natural candidates for the compact invariant sets $X_n$ in Condition \Aa.

\begin{remark}
In~\cite{PS07}, Pfister and Sullivan prove a variational principle for the topological entropy of saturated sets, which include in particular the level sets $\Kab$, under the assumption that the system in question satisfies two properties, which they call the \emph{g-almost product property} and the \emph{uniform separation property}.  Expansive systems satisfy the latter, and uniformly hyperbolic systems satisfy the former.  For such systems, they prove (among other things) the following multifractal result for any continuous $\ph$~\cite[Proposition 7.1]{PS07}:
\begin{equation}\label{eqn:PfSu}
\Ba = \htop(\Kab) = \sup\left\{ h(\mu) \,\Big|\, \mu\in \Mf(X), \int \ph\,d\mu = \alpha\right\}.
\end{equation}
Given~\eqref{eqn:PfSu}, it is not difficult to show that~\eqref{eqn:BisTL} holds, which establishes the multifractal formalism for systems with the g-almost product property and uniform separation, provided the potential is continuous.  In particular, this includes the example given above, as well as some (but by no means all) of the examples mentioned below.
\end{remark}

\subsubsection{Parabolic maps}

An important class of non-uniformly expanding maps is the Manneville--Pomeau maps, which are non-uniformly expanding interval maps with an indifferent fixed point.  The primary potential of interest in this case is the geometric potential $\log\abs{f'}$, which corresponds to studying a non-H\"older potential on a \emph{uniformly} expanding interval map via an appropriate change of coordinates; thus this is closely related to the previous example.

The thermodynamic properties and Lyapunov spectra of these maps were studied in~\cite{kN00,GR09}; once again, Theorem~\ref{thm:birkhoff} provides a direct proof of the multifractal results using the thermodynamic results, although as above, one would need to establish Condition \Aa\ to deal with the linear parts of the spectrum using Theorem~\ref{thm:phase}.  We also remark that a significant achievement of~\cite{GR09} is to deal with the endpoints of the spectrum ($\lambda=0$ and $\lambda=\infty$), which cannot be dealt with using the present results.

Moving to two (real) dimensions, let $f\colon \overline{\CC} \to \overline{\CC}$ be a parabolic rational map of the Riemann sphere; that is, a rational map such that the Julia set $J(f)$ contains at least one indifferent fixed point (that is, a fixed point $z_0$ for which $|f'(z_0)|=1$), but does not contain any critical points.  Following Makarov and Smirnov~\cite{MS00}, we say that $f$ is \emph{exceptional} if there is a finite, non-empty set $\Sigma \subset \overline{\CC}$ such that $f^{-1}(\Sigma) \setminus \Crit f = \Sigma$, where $\Crit f$ is the set of critical points of $f$.

Let $\ph(z) = \log |f'(z)|$ be the geometric potential; combining the results in~\cite{MS00} with~\cite[Corollary D.1 and Theorem G]{hH08}, we see that if $f$ is non-exceptional, then the graph of the function $\Tb$ is as shown in Figure~\ref{fig:phase-transition}.  In particular, $\Tb$ is analytic and strictly convex on $(q_0,\infty)$, where $q_0 = -\dim_H J(f)$, and so writing
\[
\alpha_1 = D^+\Tb(q_0), \qquad \alpha_2 = \lim_{q \to \infty} \Tb'(q),
\]
it follows from Theorem~\ref{thm:birkhoff} that $\Ba = \TbL$ on $(\alpha_1,\alpha_2)$.  Since we are dealing with the geometric potential, this is also the entropy spectrum for Lyapunov exponents, and we may apply~\eqref{eqn:dimlyap} to obtain the dimension spectrum for Lyapunov exponents, $\LDa = \frac 1\alpha \TbL$.

This result is obtained by other methods in~\cite{GPR09}, where it is also shown that the spectra are linear on $[0,\alpha_1]$ (the dotted line in Figure~\ref{fig:phase-transition}).  As before, giving an alternate proof of this using Theorem~\ref{thm:phase} would require establishing Condition \Aa.

Once again, Pfister and Sullivan's results establish the formalism for the Birkhoff spectrum here, but \emph{not} for the dimension spectrum for Lyapunov exponents, as they only consider topological entropy.

\subsubsection{Other non-uniformly hyperbolic systems}

The existence and uniqueness of equilibrium states for a broad class of non-uniformly expanding maps in higher dimensions was studied by Oliveira and Viana~\cite{OV08} and by Varandas and Viana~\cite{VV08}.  To the best of the author's knowledge, the multifractal properties of these systems have not been studied at all, and so they provide an ideal application of Theorem~\ref{thm:birkhoff}.  It does not appear to be known whether or not these systems, which may have contracting regions, satisfy specification or any other property that would imply Pfister and Sullivan's g-almost product property, and so the results of~\cite{PS07} cannot be applied.

We describe the systems studied in~\cite{VV08} and use the results of that paper to apply Theorem~\ref{thm:birkhoff}.  Let $M$ be a compact manifold of dimension $m$ with distance function $d$ (more generally, Varandas and Viana consider metric spaces in which the Besicovitch covering lemma holds).  Let $f\colon M\to M$ be a local homeomorphism, and let $L(x)$ be a bounded function such that for every $x\in M$ there exists a neighbourhood $U_x\ni x$ such that $f_x = f|_{U_x} \colon U_x \to f(U_x)$ is invertible, with
\[
d(f(y),f(z)) \geq \frac 1{L(x)} d(y,z)
\]
for all $y,z\in U_x$.  Thus if $L(x) <1$, then $f$ is expanding at $x$, while if $L(x)\geq 1$, then $L$ controls how much contraction can happen near $x$.

Assuming every point has finitely many preimages, we write $\deg_x(f) = \# f^{-1}(x)$.  Assume also that level sets for the degree are closed and that $M$ is connected;  then is it shown in~\cite{VV08} that up to considering some iterate $f^N$ of $f$, we can assume that $\deg_x(f)\geq e^{\htop(f)}$ for all $x$.

The final conditions on the map $f$ are as follows:  there exist constants $\sigma>1$ and $L>0$ and an open region $\AAA\subset M$ such that
\begin{enumerate}[(H1)]
\item  $L(x)\leq L$ for every $x\in \AAA$ and $L(x)\leq \sigma^{-1}$ for all $x\in M\setminus \AAA$, and $L$ is close to $1$ (see~\cite{VV08} for precise conditions).
\item  There exists $k_0\geq 1$ and a covering $\PPP = \{P_1, \dots, P_{k_0}\}$ of $M$ by domains of injectivity for $f$ such that $\AAA$ can be covered by $r<e^{\htop(f)}$ elements of $\PPP$.
\end{enumerate}

That is, $f$ is uniformly expanding outside of $\AAA$, and does not display too much contraction inside $\AAA$; furthermore, since there are at least $e^{\htop(f)}$ preimages of any given point $x$, and only $r$ of these can lie in covering of $\AAA$ by elements of $\PPP$, every point has at least one preimage in the expanding region.

The requirement on the potential $\ph$ is as follows:
\begin{enumerate}[(P)]
\item  $\ph\colon M\to \RR$ is H\"older continuous and $\sup\ph - \inf \ph < \htop(f) - \log r$.
\end{enumerate}
It is proved in~\cite{VV08} that for any map $f$ and potential $\ph$ satisfying these conditions, there exists a unique equilibrium state for $\ph$.  In particular, if (P) holds for $\ph$, then there exists $q_0>1$ such that (P) holds for $q\ph$ as well, for all $q\in (-q_0,q_0)$.  Thus Theorem~\ref{thm:birkhoff} applies, and we have the following result on the Birkhoff spectrum.

\begin{proposition}\label{prop:VV}
Given a map $f\colon M\to M$ satisfying (H1) and (H2) and a H\"older continuous potential $\ph\colon M\to \RR$ satisfying (P), there exists $q_0>1$ such that $\Tb$ is $\CCC^1$ on the interval $(-q_0,q_0)$, and writing
\[
\alpha_1 = \lim_{q\to -q_0^+} \Tb'(q), \qquad \alpha_2 = \lim_{q\to q_0^-} \Tb'(q),
\]
we have $\Ba = \TbL(\alpha)=\inf_{q\in \RR} (\Tb(q) - q\alpha)$ for every $\alpha\in [\alpha_1,\alpha_2]$.
\end{proposition}

See~\cite{VV08} for examples of specific systems to which their conditions, and hence Proposition~\ref{prop:VV}, apply.

\subsubsection{Maps with critical points}

Ever since the family of logistic maps was introduced, unimodal and multimodal maps have received a great deal of attention.   Existence and uniqueness of equilibrium states for a certain class of bounded potentials were established in~\cite{BT08}.  In particular, let $\HHH$ denote the collection of topologically mixing $\CCC^\infty$ interval maps $f\colon [0,1]\to[0,1]$ with hyperbolically repelling periodic points and non-flat critical points; given $f\in \HHH$, let $\ph\colon [0,1]\to \RR$ be a H\"older continuous potential such that
\begin{equation}\label{eqn:BR}
\sup \ph - \inf \ph < \htop (f).
\end{equation}
It is shown in~\cite{BT08} that there exists a unique equilibrium state for $\ph$, and so the analogue of Proposition~\ref{prop:VV} holds here.


In fact, it was shown by Blokh that any continuous topologically mixing interval map has the specification property (see, for example,~\cite{jB97}), which implies the g-almost product property, and so Pfister and Sullivan's result applies here, showing that the multifractal formalism holds for \emph{any} continuous potential $\ph$ on the entire spectrum.  However, their result does not apply to unbounded potentials such as the geometric potential $\ph(x) = -\log \abs{f'(x)}$.

The potentials $q\ph$, where $\ph$ is the geometric potential, were studied in~\cite{PS08,BT09,IT09}.  In the last of these papers, Iommi and Todd showed that for a related class of maps $f$, the potential $q\ph$ has a unique equilibrium state for all $q\in (-\infty, 0]$.  (In fact, they obtain results for $q>0$ as well, but we do not yet have the tools to use these here.)  Thus we may apply Theorem~\ref{thm:singularity} and show that if $\alpha_0 = \lim_{q\to 0^-} \Tb'(q)$ and $\amin = \lim_{q\to-\infty} \Tb'(q)$, then for all $q\leq 0$, we have
\[
\Tb(q) = \sup_{\alpha \in \RR} (\Ba + q\alpha),
\]
and for all $\alpha\in [\amin, \alpha_0]$, we have
\[
\Ba = \inf_{q\in \RR} (\Tb(q) - q\alpha).
\]
In particular, $\Ba$ is strictly concave and $\CCC^1$ on $[\amin,\alpha_0]$, and furthermore, $\Kab = \emptyset$ for $\alpha < \amin$.


\subsection{Entropy spectrum}

\subsubsection{Uniform hyperbolicity}

For uniformly hyperbolic systems, it can be shown that equilibrium states are Gibbs measures, and so Theorem~\ref{thm:entropy} applies to the entropy spectrum $\Ea$.  This gives an alternate proof of a particular case of the results in~\cite{TV99}, where the multifractal analysis of the entropy spectrum is carried out for expansive maps with specification (which includes uniformly hyperbolic systems).

\subsubsection{Parabolic maps}

Kesseb\"ohmer proves the existence of (non-invariant) weak Gibbs measures for continuous potentials on shift spaces~\cite{mK01}; in~\cite{JR09}, Jordan and Rams examine these weak Gibbs measures as measures on interval maps with parabolic fixed points.  Theorem~\ref{thm:entropy} then gives results regarding the entropy spectra of these measures.

\subsection{Dimension spectrum}

Conformality is automatic for one-dimensional piecewise smooth maps and for rational maps of the Riemann sphere; this provides an ideal setting to apply Theorem~\ref{thm:dimension}.

\subsubsection{Uniformly expanding maps}

For uniformly expanding maps of the interval, we have $a(x)=|f'(x)|>1$ uniformly, and so $\log a$ is positive and bounded away from $0$.  It immediately follows from the remarks in the previous section that all the conditions of Theorem~\ref{thm:dimension} are met.

The same results hold on conformal repellers in any dimension, as shown in~\cite{PW97}.  Our proof here provides an alternate proof of some of the results in that paper.

\subsubsection{Parabolic maps}

The dimension spectrum for Manneville--Pomeau maps has been studied in~\cite{kN00,JR09}; once again, the present approach provides an alternate proof of some results.

\subsubsection{Maps with critical points}

Given a multimodal map $f\in \HHH$, the multifractal analysis of the dimension spectrum for Gibbs measures associated to the potentials described above is carried out in~\cite{mT08,IT09}.  At present, these results \emph{cannot} be obtained using the results in this paper, due to the presence of the critical point, which the tools used here cannot yet handle.

\section{Preparatory results}\label{sec:prep}

\subsection{Convergence results}

\begin{proposition}\label{prop:convergence}
Let $X$ be a compact metric space, $f\colon X\to X$ be continuous, and $\ph\in \Af$.  Let $\mu\in \Mf(X)$ be an invariant measure, and consider a sequence of (not necessarily $f$-invariant) measures $\{\mu_n\} \subset \MMM(X)$ such that $\mu_n\to \mu$ in the weak* topology.  Then
\begin{equation}\label{eqn:Aftoo}
\lim_{n\to\infty} \int \ph \,d\mu_n = \int \ph \,d\mu.
\end{equation}
\end{proposition}
\begin{proof}
If $\ph$ is continuous, then this is immediate.  If $\ph$ is discontinuous, then let $M\in \RR$ be such that $|\ph(x)| \leq M$ for all $x\in X$, and fix $\eps>0$.  Condition (B) in the definition of $A_f$ tells us that $\mu(\ccph) = 0$, and thus there exists an open neighbourhood $B \supset \ccph$ such that $\mu(\overline{B}) < \eps$.  Since $\overline{B}$ is closed, we have
\[
\mu(\overline{B}) \geq \ulim_{n\to\infty} \mu_n(\overline{B}),
\]
and so there exists $N$ such that $\mu_n(B) \leq \mu_n(\overline{B}) < 2\eps$ for all $n\geq N$.  Now we have
\[
\left\lvert \int_X \ph \,d\mu - \int_X \ph \,d\mu_n \right\rvert \leq
\left\lvert \int_{X\setminus B} \ph \,d\mu - \int_{X\setminus B} \ph \,d\mu_n \right\rvert + 
\left\lvert \int_B \ph \,d\mu - \int_B \ph \,d\mu_n \right\rvert.
\]
Since $\ph$ is continuous on the compact set $X\setminus B$, the first difference goes to $0$ as $n\to \infty$.  Furthermore, by the above estimates, the second difference is less than $3M\eps$.  Since $\eps>0$ was arbitrary, this completes the proof of~\eqref{eqn:Aftoo}.
\end{proof}

\begin{proposition}\label{prop:singularity}
Let $X$ be a compact metric space and $\psi\colon X\to \RR\cup\{-\infty\}$ be continuous where finite (and hence bounded above).  Consider a sequence of measures $\{\mu_n\}$ converging to $\mu$ in the weak* topology, and suppose that $\int\psi\,d\mu > -\infty$.  Then
\begin{equation}\label{eqn:sc}
\int \psi\,d\mu \geq \ulim_{n\to\infty} \int \psi\,d\mu_n.
\end{equation}
\end{proposition}
\begin{proof}
Given $M<0$, define a continuous function $\psi_M\colon X\to \RR$ by
\[
\psi_M(x) = \max(\psi(x),M).
\]
Because $\psi$ is integrable with respect to $\mu$, we have for every $\eps>0$ some $M<0$ such that
\[
\int (\psi_M - \psi) \,d\mu < \eps,
\]
from which we deduce that
\[
\int \psi\,d\mu \geq \int \psi_M \,d\mu - \eps = \lim_{n\to\infty} \int \psi_M \,d\mu_n - \eps
\geq \ulim_{n\to\infty} \int \psi \,d\mu_n - \eps.
\]
Because $\eps>0$ was arbitrary, this establishes~\eqref{eqn:sc}.
\end{proof}

Observe that there are no dynamics in Proposition~\ref{prop:singularity}, so there is no requirement that any of the measures $\mu_n$ or $\mu$ be invariant.

\subsection{Measures associated with approximate level sets}

Recall that the level sets $\Kab$ are defined by
\[
\Kab(\ph) = \left\{ x\in X \,\Big|\, \lim_{n\to\infty} \frac 1n S_n\ph(x) = \alpha \right\},
\]
where we write $\Kab(\ph)$ to emphasise the role of the potential function $\ph$.  This may be rewritten as
\begin{align*}
\Kab(\ph) &= \left\{ x\in X \,\Big|\, \forall \eps>0 \exists N \text{ such that } 
\left| \frac 1n S_n\ph(x) - \alpha \right| < \eps \text{ for all } n\geq N \right\} \\
&= \bigcap_{\eps>0} \bigcup_{N\in\NN} \bigcap_{n\geq N} \left\{ x\in X\,\Big|\, \left| \frac 1n S_n\ph(x) - \alpha \right| < \eps \right\}.
\end{align*}
In the proofs of our main results, we will need to consider the following ``approximate level sets'':
\begin{equation}\label{eqn:Faen}
\begin{aligned}
\Faen(\ph) &= \bigcap_{n\geq N} \left\{ x\in X\,\Big|\, \left| \frac 1n S_n\ph(x) - \alpha \right| < \eps \right\} \\
\Fae(\ph) &= \bigcup_{N\in\NN} \Faen(\ph).
\end{aligned}
\end{equation}
For these we have
\[
\Kab(\ph) = \bigcap_{\eps>0} \Fae(\ph),
\]
In particular, the following relations will be quite useful:
\begin{align*}
\htop \Kab(\ph) &\leq \htop \Fae(\ph) = \sup_N \left(\htop \Faen(\ph)\right), \\
\dim_H \Kab(\ph) &\leq \dim_H \Fae(\ph) = \sup_N \left(\dim_H \Faen(\ph)\right).
\end{align*}
Observe that for a continuous function $\ph$, each set $\{x\in X \mid |(1/n)S_n\ph(x) - \alpha|<\eps\}$ is a union of intervals, and $\Faen(\ph)$ is a countable intersection of such sets.  When $Z$ is such a set, it is reasonable to approximate $\htop Z$ with $\lhtop Z$, which gives us an upper bound.  A similar upper bound applies when we study the topological pressure.

The utility of the capacity quantities (entropy and pressure) for our purposes is in the following lemma, which shows that when we deal with sets like $\Faen$ on which the Birkhoff averages converge \emph{uniformly} to a given range of values, then we can build measures with large free energy and with the expected integrals.

\begin{lemma}\label{lem:buildmeasure}
Let $X$ be a compact metric space, $f\colon X\to X$ be continuous, and $\psi, \zeta \in \Af$.  Fix $Z\subset X$ and let $\beta_1,\beta_2\in [-\infty,\infty]$ be given by
\[
\beta_1 = \llim_{n\to\infty} \inf_{x\in Z} \frac 1n S_n \psi(x), \qquad
\beta_2 = \ulim_{n\to\infty} \sup_{x\in Z} \frac 1n S_n \psi(x).
\]
Then for every $\gamma>0$ there exists $\mu\in\Mf(X)$ satisfying the following:
\begin{align}
\label{eqn:stillhere}
\int \psi \,d\mu &\in [\beta_1,\beta_2], \\
\label{eqn:varprinc}
h(\mu) + \int \zeta \,d\mu &\geq \uP_Z(\zeta) - \gamma.
\end{align}
\end{lemma}
\begin{proof}
The construction of $\mu$ satisfying~\eqref{eqn:varprinc} is given in part 2 of the proof of~\cite[Theorem 9.10]{pW75}, and goes as follows.  Choose $\delta>0$ such that
\[
\ulim_{n\to\infty} \frac 1n\sum_{x\in E_n} e^{S_n \zeta(x)} > \uP_Z(\zeta) - \gamma,
\]
where $E_n$ is a maximal $(n,\delta)$-separated set, and define an atomic measure $\sigma_n$ on $E_n$ by
\begin{equation}\label{eqn:sigman}
\sigma_n = \frac{ \sum_{y\in E_n} e^{S_n \zeta(y)} \delta_y } { \sum_{z\in E_n} e^{S_n \zeta(z) } }.
\end{equation}
Define $\mu_n$ by
\begin{equation}\label{eqn:mun}
\mu_n = \frac 1n \sum_{i=0}^{n-1} \sigma_n \circ f^{-i},
\end{equation}
and let $\mu$ be any weak* limit of the sequence $\mu_n$---then $\mu$ is invariant, and the estimate~\eqref{eqn:stillhere} follows from Proposition~\ref{prop:convergence} upon observing that for every $\eps>0$, there exists $N\in \NN$ such that $\int \psi\,d\mu_n \in [\beta_1-\eps, \beta_2+\eps]$ for all $n\geq N$.  The estimate~\eqref{eqn:varprinc} is shown in the proof in~\cite{pW75}; although the proof there assumes that $\zeta$ is continuous, this is only used to guarantee the convergence $\int \zeta \,d\mu_{n_j} \to \int \zeta \,d\mu$, which in our case is given by Proposition~\ref{prop:convergence}.
\end{proof}

The full strength of Lemma~\ref{lem:buildmeasure} is only needed in the proof of Theorem~\ref{thm:dimension} (for the dimension spectrum).  For the proof of Theorem~\ref{thm:birkhoff} (for the Birkhoff spectrum), we only need the case $\zeta=0$.  In particular, in order to prove Theorems~\ref{thm:high-entropy} and~\ref{thm:singularity}, we only need the following two versions of Lemma~\ref{lem:buildmeasure}.

\begin{lemma}\label{lem:high-entropy}
Let $X$ be a compact metric space, $f\colon X\to X$ be continuous, and $\ph\colon X\to \RR$ be Borel measurable and bounded above and below.  Suppose $Z\subset X$ is such that $\lhtop(Z) > \uhtop(\CCC(\ph))$.  Fix $Z\subset X$ and let $\beta_1,\beta_2\in [-\infty,\infty]$ be given by
\[
\beta_1 = \llim_{n\to\infty} \inf_{x\in Z} \frac 1n S_n \ph(x), \qquad
\beta_2 = \ulim_{n\to\infty} \sup_{x\in Z} \frac 1n S_n \ph(x).
\]
Then for every $\gamma>0$ there exists $\mu\in\Mf(X)$ satisfying the following:
\begin{align}
\label{eqn:stillhere2}
\int \ph \,d\mu &\in (\beta_1 - \gamma,\beta_2 + \gamma), \\
\label{eqn:varprinc2}
h(\mu) &\geq \lhtop(Z) - \gamma.
\end{align}
\end{lemma}
\begin{proof}
For $n\in \NN$ and $\delta>0$, let $P_n^\delta$ be the maximal cardinality of an $(n,\delta)$-separated subset of $Z$, and recall that
\[
\lhtop(Z) = \lim_{\delta\to 0} \llim_{n\to\infty} \frac 1n \log P_n^\delta.
\]
In particular, decreasing $\gamma$ if necessary, we may choose $\delta>0$ such that
\begin{equation}\label{eqn:entropies}
\uhtop(\CCC(\ph),\delta) < \lhtop(Z) - \gamma < \llim_{n\to\infty} \frac 1n \log P_n^\delta.
\end{equation}
Writing $h_0 = \llim_{n\to\infty} \frac 1n\log P_n^\delta$, we choose $\eta>0$ such that $h_0 - \eta > \uhtop(\CCC(\ph),\delta)$.  Thus there exists $C>0$ such that for every $m\in \NN$ there exists a set $F_m\subset \CCC(\ph)$ such that $\# F_m \leq Ce^{m(h_0-\eta)}$ and $U_m = \bigcup_{x\in F_m} B(x,m,\delta) \supset \CCC(\ph)$.  Observe that $U_m$ is open because $f$ is continuous.

Given $n\in \NN$, let $E_n$ be an $(n,\delta)$-separated subset of $Z$ with maximum cardinality $\#E_n = P_n^\delta$.  Following the previous proof, consider the measures $\sigma_n$ given by~\eqref{eqn:sigman} with $\zeta=0$:
\begin{equation}\label{eqn:sigmanm}
\sigma_n = \frac{\sum_{x\in E_n} \delta_x}{ \#E_n}.
\end{equation}
Now we vary the construction slightly; given $0\leq m<n$, we go $n-m$ steps (not $n$) along each orbit:
\begin{equation}\label{eqn:munm}
\mu_n^m = \frac 1n \sum_{k=0}^{n-m-1} \sigma_n \circ f^{-k}.
\end{equation}
That is, $\mu_n^m$ is a convex combination of $\delta$-measures evenly distributed across the first $n - m$ points in each orbit that begins in $E_n$.

For every $0\leq k < n - m -1$, consider the set
\[
B_n^m(k) = \{x\in E_n \mid f^k \in U_m\} = \bigcup_{z\in F_m} f^{-k}(B(z,m,\delta)) \cap E_n.
\]
Observe that for every $z\in F_m$ and every pair $x\neq y\in f^{-k}(B(z,m,\delta)) \cap E_n$, we have $d(f^i(x), f^i(y)) < \delta$ for all $n - m \leq i < n$, and since $E_n$ is $(n,\delta)$-separated, it follows that $d(f^i(x),f^i(y))\geq \delta$ for some $0\leq i <n - m$.  In particular, $f^{-k}(B(z,m,\delta)) \cap E_n \subset Z$ is $(n - m,\delta)$-separated, and hence has cardinality at most $P_{n-m}^\delta$.  It follows that
\[
\# B_n^m(k) \leq C e^{m(h_0-\eta)} P_{n-m}^\delta,
\]
and hence
\[
\sigma_n(f^{-k}(U_m)) = \frac{\#B_n^m(k)}{\#E_n} \leq C e^{-\eta m + mh_0} \frac{ P_{n-m}^\delta}{P_n^\delta}.
\]
This holds for all $0\leq k< n - m$, and hence
\begin{equation}\label{eqn:badset}
\mu_n^m (U_m) \leq C e^{-\eta m} \frac{ e^{mh_0} P_{n-m}^\delta}{P_n^\delta}.
\end{equation}

Thus in order to bound $\mu_n^m(U_m)$, we need some control of the ratio $P_{n-m}^\delta/P_n^\delta$.  Observe that if $P_n^\delta$ is actually equal to $e^{nh_0}$ for all $n$, then~\eqref{eqn:badset} immediately yields the bound $\mu_n^m (U_m) \leq C e^{-\eta m}$.  However, $P_n^\delta$ may not grow as uniformly as we would like, so we must be more careful.

Given $m\in \NN$, consider the quantity
\[
L(m) = \ulim_{n\to\infty} (\log (P_n^\delta) - \log (P_{n-m}^\delta) - mh_0).
\]
Suppose $L(m) < 0$.  Then there exists $\eps>0$ and $N\in \NN$ such that for all $n\geq N$, we have
\[
\log (P_n^\delta) - \log (P_{n-m}^\delta) - mh_0 < -\eps.
\]
In particular, this gives the following for every $k\in \NN$:
\[
\log (P_{N+km}^\delta) < \log (P_N^\delta) + kmh_0 - k\eps.
\]
Dividing by $km$ and taking the limit as $k\to\infty$, we get
\[
h_0 = \llim_{n\to\infty} \frac 1n \log (P_n^\delta) \leq 
\llim_{k\to\infty} \frac 1{N+km} (P_{N+km}^\delta) < h_0 - \frac \eps m,
\]
a contradiction.  This proves that $L(m) \geq 0$, from which we deduce that for every $m\in \NN$, there exists a sequence $n_j = n_j(m) \to\infty$ such that
\[
\llim_{j\to\infty} (\log (P_{n_j}^\delta) - \log (P_{n_j-m}^\delta) - m h_0) \geq 0,
\]
or equivalently,
\begin{equation}\label{eqn:Pnj}
\llim_{j\to\infty} \frac {P_{n_j}^\delta}{e^{mh_0} P_{n_j-m}^\delta} \geq 1.
\end{equation}
In combination with~\eqref{eqn:badset}, this will soon give us the bound we need.

As in the proof of Lemma~\ref{lem:buildmeasure}, let $\mu^m$ be a weak* limit point of the sequence $\mu_{n_j}^m$ (by passing to a subsequence if necessary, we assume that $\mu_{n_j}^m \to \mu^m$).  Invariance of $\mu^m$ and the entropy estimate~\eqref{eqn:varprinc2} hold just as before, so it only remains to show~\eqref{eqn:stillhere2}.

Let $M = \sup_{x\in X} \abs{\ph(x)}$, and choose $m$ large enough so that $C e^{-\eta m} <\gamma/2M$.  Carry out the above construction for this value of $m$, and observe that because $U_m$ is open, we have
\begin{equation}\label{eqn:smallbad}
\mu^m(U_m) \leq \ulim_{n_j\to\infty} \mu_{n_j}^m(U_m) \leq C e^{-\eta m} < \frac \gamma{2M},
\end{equation}
where the middle inequality follows from~\eqref{eqn:badset} and~\eqref{eqn:Pnj}.  Consequently,
\begin{multline}
\left\lvert \int_X \ph \,d\mu^m - \int_X \ph \,d\mu_{n_j}^m \right\rvert \leq \\
\left\lvert \int_{X\setminus U_m} \ph \,d\mu^m - \int_{X\setminus U_m} \ph \,d\mu_{n_j}^m \right\rvert + 
\left\lvert \int_{U_m} \ph \,d\mu^m - \int_{U_m} \ph \,d\mu_{n_j}^m \right\rvert.
\end{multline}
Since $\ph$ is continuous on the compact set $X\setminus U_m$, the first difference goes to $0$ as $j\to\infty$, and by~\eqref{eqn:smallbad}, the second term is less than $\gamma$; this proves \eqref{eqn:stillhere2} for $\mu^m$.
\end{proof}

\begin{lemma}\label{lem:singularity}
Let $X$ be a compact metric space, let $f\colon X\to X$ be continuous, and let $\psi\colon X\to \RR\cup\{-\infty\}$ be continuous where finite (and hence bounded above).  Fix $Z\subset X$ and let $\beta\in \RR$ be given by
\[
\beta = \llim_{n\to\infty} \inf_{x\in Z} \frac 1n S_n \psi(x).
\]
Then for every $\gamma>0$ there exists $\mu\in\Mf(X)$ satisfying the following:
\begin{align}
\label{eqn:stillhere3}
\int \psi \,d\mu &\geq \beta, \\
\label{eqn:varprinc3}
h(\mu) &\geq \uhtop(Z) - \gamma.
\end{align}
\end{lemma}
\begin{proof}
The proof is exactly as in Lemma~\ref{lem:buildmeasure} with the choice $\zeta = 0$, $\eta = \psi$, with Proposition~\ref{prop:singularity} taking the place of Proposition~\ref{prop:convergence}.
\end{proof}

\section{Proof of Theorem~\ref{thm:birkhoff}}

The proof of Theorem~\ref{thm:birkhoff} proceeds in three parts, corresponding to the three parts of the theorem.  In the first part, we show that $\Tb$ is the Legendre transform of $\BBB$, thus establishing~\eqref{eqn:TisBL}.  From this, it immediately follows by standard properties of the Legendre transform that $\TbL$ is the concave hull of $\BBB$; that is, it is the smallest concave function greater than or equal to $\BBB$ at all $\alpha$.

Part II of the theorem is an easy consequence of the following proposition.

\begin{proposition}\label{prop:empty}
Suppose that $\Kab$ is non-empty; that is, there exists $x\in X$ such that $\ph^+(x)=\lim_{n\to\infty}\frac1n S_n\ph(x) = \alpha$.  Then $P^*(q\ph)\geq \alpha q$ for all $q\in \RR$.
\end{proposition}

Once Part I is established, Part III of the theorem is proved via the following series of intermediate results.

\begin{proposition}\label{prop:concave}
Let $\ph$ be Borel measurable and suppose that $\nu_q$ is an ergodic equilibrium state for $q\ph$.  Let $\alpha = \int \ph\,d\nu_q$.  Then
\begin{equation}\label{eqn:BisTLpt}
\Ba \geq \TbL(\alpha).
\end{equation}
\end{proposition}

Note the requirement in Proposition~\ref{prop:concave} that the equilibrium state $\nu_q$ be ergodic.  It will often be the case that general arguments will give the existence of \emph{non}-ergodic equilibrium states with $\alpha(\nu_q) = \alpha$, but this is not sufficient for our purposes.

\begin{proposition}[Ruelle's formula for the derivative of pressure]\label{prop:ruelle}
Let $\psi$ and $\phi$ be Borel measurable functions.  If the function
\[
q\mapsto P^*(\psi + q\phi)
\]
is differentiable at $q$, and if in addition $\nu_q$ is an equilibrium state for $\psi + q\phi$, then
\begin{equation}\label{eqn:derivative}
\frac{d}{dq} P^*(\psi + q\phi) = \int_X \phi \,d \nu_q.
\end{equation}
\end{proposition}

\begin{corollary}\label{cor:alpha-interval}
Suppose $\Tb$ is continuously differentiable on $(q_1,q_2)$ and $q\ph$ has an equilibrium state $\nu_q$ for each $q\in(q_1,q_2)$.  Let $\alpha_1 = D^+\Tb(q_1)$ and $\alpha_2 = D^-\Tb(q_2)$; then for every $\alpha\in(\alpha_1,\alpha_2)$ there exists $q\in \RR$ such that $q\ph$ has an ergodic equilibrium state $\nu_q$ with $\alpha = \int \ph\,d\nu_q$.
\end{corollary}

Once these results are established,~\eqref{eqn:BisTL} is a direct consequence of Proposition~\ref{prop:concave} and Corollary~\ref{cor:alpha-interval}.  It then follows from basic properties of the Legendre transform that $\BBB = \TbL$ has the same regularity as $\Tb$ (except for values of $\alpha$ corresponding to intervals on which $\Tb$ is affine).

\begin{proof}[Proof of part I]
We prove~\eqref{eqn:TisBL} by establishing the following two inequalities:
\begin{align}
\Tb &\leq \BL, \label{eqn:TleqBL} \\
\Tb &\geq \BL. \label{eqn:BleqTL}
\end{align}

First we prove~\eqref{eqn:TleqBL}.  Recall that
\[
\Tb(q) = P^* (q\ph) = \sup_{\nu\in\Mfe(X)} \left\{h(\nu) + q \int_X \ph \,d\nu \right\}.
\]
By Birkhoff's ergodic theorem, every ergodic measure $\nu$ has $\nu(\Kab)=1$ for some $\alpha$, and so for $\nu$-almost every $x\in \Kab$ (in particular, for \emph{some $x\in \Kab$}), we have $\int_X \ph\,d\nu = \ph^+(x) = \alpha$.  It follows that
\begin{align*}
\Tb(q) &= \sup_{\alpha\in\RR} \left( 
\sup_{\nu \in\Mfe(\Kab)} \left\{h(\nu) + q \int_X \ph \,d\nu \right\} \right) \\
&\leq \sup_{\alpha\in\RR} \left(\htop(\Kab) + q\alpha\right) = \BL(q),
\end{align*}
where the inequality $h(\nu)\leq \htop(\Kab)$ follows from Theorem~A2.1 in~\cite{yP98}.

Now we prove the reverse inequality~\eqref{eqn:BleqTL}, by showing that $\Tb(q) = P^*(q\ph) \geq \Ba + q\alpha$ for all $q,\alpha\in \RR$.  To this end, we fix $\eps>0$ and consider the sets $\Fae$, $\Faen$ defined in~\eqref{eqn:Faen}.

Applying Lemma~\ref{lem:buildmeasure} with $\zeta = 0$, $\psi = \ph$, $Z = \Faen$, and some $\gamma>0$, we obtain a measure $\mu\in \Mf(X)$ with $h(\mu) \geq \lhtop(\Faen)-\gamma$ and $\int \ph \,d\mu \geq \alpha - \eps$.  It follows that
\[
P^*(q\ph) \geq h(\mu) + q \int \ph \,d\mu \geq \lhtop(\Faen) - \gamma + q\alpha - q\eps,
\]
and since Lemma~\ref{lem:buildmeasure} can be applied with arbitrarily small $\gamma$, we get
\[
P^*(q\ph) \geq \htop(\Faen) + q\alpha - q\eps.
\]
Taking the supremum over all $N$ yields
\[
P^*(q\ph) \geq \htop(\Fae) + q\alpha - q\eps \geq \htop(\Kab) + q\alpha - q\eps,
\]
and since $\eps>0$ was arbitrary, this implies
\[
P^*(q\ph) \geq \htop(\Kab) + q\alpha.
\]
This holds for all $q,\alpha\in \RR$, which establishes~\eqref{eqn:BleqTL}.
\end{proof}

We now proceed to the proof of Part II.

\begin{proof}[Proof of Proposition~\ref{prop:empty}]
Suppose $\alpha\in \RR$ is such that there exists $x\in \Kab$.  Consider the empirical measures
\[
\mu_{n,x} = \sum_{i=0}^{n-1} \delta_{f^i(x)}.
\]
Choose any subsequence $n_k$ such that $\mu_{n_k,x}$ converges in the weak* topology to some $\mu \in \MMM^f(X)$.  Then by Proposition~\ref{prop:convergence}, we have $\int \ph\,d\mu =  \alpha$, and in particular,
\[
P^*(q\ph) \geq h(\mu) + \int q\ph\,d\mu \geq q\int \ph\,d\mu \geq q\alpha
\]
for every $q\in \RR$.
\end{proof}

Finally, we prove the string of propositions which implies Part III.

\begin{proof}[Proof of Proposition~\ref{prop:concave}]
Observe that since $\nu_q$ is ergodic, we have $\nu_q(\Kab)=1$, and hence $h(\nu_q) \leq \htop(\Kab)$.  Thus
\begin{align*}
\TbL(\alpha) &= \inf_{q'\in \RR} (\Tb(q') - q'\alpha') \\
&\leq \Tb(q) - q\alpha' = P^*(q\ph) - q\alpha \\
&= h(\nu_q) + \int_X q\ph \,d\nu_q - q\alpha \\
&\leq \htop(\Kab) = \Ba.\qedhere
\end{align*}
\end{proof}

\begin{proof}[Proof of Proposition~\ref{prop:ruelle}]
Write $g(q') = P^*(\psi + q'\phi)$.  Then for all $q'\in\RR$, we have
\begin{align*}
g(q') &= P^*(\psi + q'\phi) \\
&= \sup_\nu \left\{ h(\nu) + \int_X \psi \,d\nu + \int_X q'\phi \,d\nu \right\} \\
&\geq h(\nu_q) + \int_X \psi \,d\nu_q + q' \int_X \phi \,d\nu_q \\
&= P^*(\psi + q\phi) + (q'-q) \int_X \phi \,d\nu_q, \\
&= g(q) + (q'-q) \int_X \phi \,d\nu_q,
\end{align*}
whence
\[
g(q') - g(q) \geq (q'-q) \int_X \phi \,d\nu_q.
\]
In particular, for $q'>q$, we get
\[
\frac{g(q') - g(q)}{q'-q} \geq \int_X \phi \,d\nu_q,
\]
and hence $g'(q)\geq \int_X \phi \,d\nu_q$ (recall that differentiability of $g$ was one of the hypotheses of the theorem), while for $q'<q$,
\[
\frac{g(q') - g(q)}{q'-q} \leq \int_X \phi \,d\nu_q,
\]
and hence $g'(q)\leq \int_X \phi \,d\nu_q$, which establishes equality.
\end{proof}

\begin{proof}[Proof of Corollary~\ref{cor:alpha-interval}]
Since $\Tb'$ is continuous, the Intermediate Value Theorem implies that for every such $\alpha$ there exists $q$ such that $\Tb'(q)=\alpha$.  Thus applying Proposition~\ref{prop:ruelle} with $\psi=0$ and $\phi=\ph$, we see that any equilibrium state $\nu$ for $q\ph$ has $\nu(q\ph)=\alpha$.  Choose some such $\nu$; if $\nu$ is not ergodic, then any element in its ergodic decomposition is also an equilibrium state, and we are done.
\end{proof}

\section{Proof of Theorems~\ref{thm:phase},~\ref{thm:high-entropy}, and~\ref{thm:singularity}}

Given the proof of Theorem~\ref{thm:birkhoff} in the previous section, the proofs of Theorems~\ref{thm:phase},~\ref{thm:high-entropy}, and~\ref{thm:singularity} are relatively straightforward.

\begin{proof}[Proof of Theorem~\ref{thm:phase}]
Recall that the first two parts of Theorem~\ref{thm:birkhoff} hold without any assumptions on $f$, and thus we already have $\Tb = \BL$.  It remains only to show that $\Ba \geq \TbL(\alpha)$ for every $\alpha \in [\amin,\amax]$, given Condition \Aa.

Given such an $\alpha$, if there exists $q\in \RR$ such that $\Tb'(q)=\alpha$, then the proof of Theorem~\ref{thm:birkhoff} shows that $\Ba = \TbL(\alpha)$.  Thus we suppose that no such $q$ exists; in this case, let $q_0=Q(\alpha)$ be the (unique) value of $q$ such that
\[
\Tb(q) \geq \Tb(q_0) + (q-q_0)\alpha
\]
for all $q\in \RR$.  (Equivalently, we have $q_0 = -(\TbL)'(\alpha)$.)

Applying Theorem~\ref{thm:birkhoff} to the subsystem $X_n$, we see that
\[
\htop (\Kab \cap X_n) = \inf_{q\in \RR} (P_{X_n}^*(q\ph) - q\alpha);
\]
since $q\mapsto P_{X_n}^*(q\ph)$ is assumed to be differentiable on $\RR$, for every $\alpha\in [\amin,\amax]$ there exists $q_n \in \RR$ such that $A_n(q_n) = \frac d{dq} P_{X_n}^*(q\ph) |_{q=q_n} = \alpha$.  Let $\mu_n$ be an ergodic equilibrium state for $q_n\ph$ on $X_n$; then $\int \ph \,d\mu_n = \alpha$ by Proposition~\ref{prop:ruelle}, and so $\mu_n(\Kab) = 1$.  Thus we have
\begin{equation}\label{eqn:lowerbound}
\htop \Kab \geq h(\mu_n) = P_{X_n}^*(q_n\ph) - q_n\alpha.
\end{equation}

It follows from convexity of the pressure function that $q_n\to q_0$ as $n$ goes to $\infty$, and by continuity of the pressure function and Condition \Aa, this implies that
\[
\lim_{n\to\infty} P_{X_n}^*(q_n\ph) = P^*(q_0\ph),
\]
which together with~\eqref{eqn:lowerbound} shows that $\Ba \geq \Tb(q_0) - q_0 \alpha \geq \TbL(\alpha)$.
\end{proof}

\begin{proof}[Proof of Theorem~\ref{thm:high-entropy}]
The proof of Theorem~\ref{thm:high-entropy} mirrors the proof of Theorem~\ref{thm:birkhoff}; the primary difference is that Lemma~\ref{lem:high-entropy} replaces Lemma~\ref{lem:buildmeasure} in the proof of Part I, where we show~\eqref{eqn:TisBL2}.

The proof of Proposition~\ref{prop:empty} does not go through in this setting, and so Part II is weakened from the corresponding statement in Theorem~\ref{thm:birkhoff}.

The series of propositions in Part III goes through unchanged, as Proposition~\ref{prop:concave}, Proposition~\ref{prop:ruelle}, and Corollary~\ref{cor:alpha-interval} all hold without regard to continuity of the potential $\ph$.

Observe that~\eqref{eqn:TleqBL} holds here as well without modification, since its proof does not require any hypotheses on $\ph$.  Thus to prove~\eqref{eqn:TisBL2}, it suffices to establish the following inequality for every $q\in I_Q(h_0)$:
\begin{equation}\label{eqn:BleqTL2}
\Tb(q) \geq \sup_{\alpha\in I_A(h_0)} (\Ba + q\alpha).
\end{equation}
That is, we show that $\Tb(q) = P^*(q\ph) \geq \Ba + q\alpha$ for all $q \in I_Q(h_0)$ and $\alpha \in I_A(h_0)$.  Observe that if $\Ba \leq h_0$, then since $\alpha\in I_A(h_0)$ we have $\TbL(\alpha) = \inf_{q\in\RR} (\Tb(q) - q\alpha) > h_0 \geq \Ba$, and so in particular $\Tb(q) \geq \Ba + q\alpha$ for $q\in I_Q(h_0)$.  Thus it remains only to consider the case $\Ba > h_0$.

As in the proof of~\eqref{eqn:BleqTL} in Theorem~\ref{thm:birkhoff}, we fix $\eps>0$ and consider the sets $\Fae$, $\Faen$ defined in~\eqref{eqn:Faen}.  Because $h_0 < \Ba = \htop\Kab \leq \htop\Fae = \sup_N \htop\Faen$, we can find $N\in \NN$ such that $\htop\Faen > h_0$, and then apply Lemma~\ref{lem:high-entropy} with $\psi = \ph$, $Z = \Faen$, and some $\gamma>0$ to obtain a measure $\mu$ with $h(\mu) \geq \lhtop(\Faen)-\gamma$ and $\int \ph \,d\mu \geq \alpha - \eps - \gamma$.  It follows that
\[
P^*(q\ph) \geq h(\mu) + q \int \ph \,d\mu \geq \lhtop(\Faen) - \gamma + q\alpha - q(\eps + \gamma),
\]
and since Lemma~\ref{lem:high-entropy} can be applied with arbitrarily small $\gamma$, we get
\[
P^*(q\ph) \geq \htop(\Faen) + q\alpha - q\eps.
\]
Taking the supremum over all such $N$ yields
\[
P^*(q\ph) \geq \htop(\Fae) + q\alpha - q\eps \geq \htop(\Kab) + q\alpha - q\eps,
\]
and since $\eps>0$ was arbitrary, this implies
\[
P^*(q\ph) \geq \htop(\Kab) + q\alpha.
\]
This holds for all $q\in I_Q(h_0)$ and $\alpha\in I_A(h_0)$, which establishes~\eqref{eqn:BleqTL2}.

For Part II of Theorem~\ref{thm:high-entropy}, we observe that if $\Ba > h_0$, then we can apply Lemma~\ref{lem:high-entropy} exactly as above to obtain $\Tb(q) \geq \Ba + q\alpha$ for all $q\in \RR$, and hence $\TbL(\alpha) \geq \Ba > h_0$ as well, so $\alpha\in I_A(h_0)$.

As remarked above, the propositions in Part III go through unchanged, and we are done.
\end{proof}

\begin{proof}[Proof of Theorem~\ref{thm:singularity}]
The proof of Parts I of Theorem~\ref{thm:singularity} is nearly identical to the proof of Theorem~\ref{thm:high-entropy}, with Lemma~\ref{lem:singularity} replacing Lemma~\ref{lem:high-entropy} in the proof of~\eqref{eqn:TisBL}, and with $(-\infty,0)$ and $(-\infty, \alpha_0]$ replacing $I_Q(h_0)$ and $I_A(h_0)$.

Part II of Theorem~\ref{thm:singularity} follows from the observation that Proposition~\ref{prop:empty} \emph{does} apply in this setting as follows:  if $\Kab$ is non-empty for some $\alpha\in \RR$, then $P^*(q\ph) \geq \alpha q$ for all $q\leq 0$.  The proof only requires replacing Proposition~\ref{prop:convergence} with Proposition~\ref{prop:singularity}.

Again,~\eqref{eqn:TleqBL} holds here without modification.  Furthermore, for any $q\leq 0$ and $\alpha\in \RR$, we may fix $\eps>0$ and apply Lemma~\ref{lem:singularity} with $\psi = q\ph$, $Z = \Faen$, and some $\gamma>0$ to obtain a measure $\mu\in \Mf(X)$ with $h(\mu) \geq \lhtop(\Faen)-\gamma$ and $\int q\ph \,d\mu \geq q\alpha - q\eps$.  It follows that
\[
P^*(q\ph) \geq h(\mu) + \int q\ph \,d\mu \geq \lhtop(\Faen) - \gamma + q\alpha - q\eps,
\]
and just as in the proof of Theorem~\ref{thm:birkhoff}, we obtain
\[
P^*(q\ph) \geq \htop(\Kab) + q\alpha.
\]
This holds for all $q\leq 0$ and $\alpha\in\RR$, which establishes~\eqref{eqn:BleqTL}.

Part III is once again just as before.
\end{proof}

\section{Proof of Theorem~\ref{thm:dimension}}\label{sec:last}

As in the proof of Theorem~\ref{thm:birkhoff}, we carry out the proof of Theorem~\ref{thm:dimension} in three parts.  First, we show that $\Td$ is the Legendre transform of $\DDD$, establishing~\eqref{eqn:TisDL}.  From this, it immediately follows by standard properties of the Legendre transform that $\TdL$ is the concave hull of $\DDD$.

Part II of the theorem is an easy consequence of the following proposition.

\begin{proposition}\label{prop:emptyD}
Given $\alpha\in \RR$, suppose that $\Kad\cap X'$ is non-empty; that is, there exists $x\in X'$ such that $d_\mu(x) = \alpha$.  Then $\Td(q) \geq -\alpha q$ for all $q\in \RR$.  Furthermore, if there exists $x\in X'$ such that $d_\mu(x) = +\infty$, then $\Td(q) = +\infty$ for all $q<0$.
\end{proposition}

Part III of the theorem is once again proved via intermediate results similar in spirit to those in the proof of Theorem~\ref{thm:birkhoff}.

\begin{proposition}\label{prop:concaveD}
Given $q\in \RR$, let $q_n\to q$ and $t_n \to \Td(q)$ be such that $t_n\leq \Td(q_n)$ for all $n$.  Fix $\alpha\in \RR$, and suppose that for all $n\in \NN$, there exists an ergodic equilibrium state $\nu_n$ for $\ph_{q_n,t_n}$ such that $\lambda(\nu_n)>0$ and
\begin{equation}\label{eqn:alpha-integralD}
\alpha=\frac{-\int \ph_1 \,d\nu_n}{\lambda(\nu_n)}.
\end{equation}
Then $\Da \geq \TdL(\alpha)$.
\end{proposition}

\begin{proposition}\label{prop:alpha-intervalD}
Given $\eta>0$ and $I_Q=(q_1,q_2)$, suppose that the map $(q,t) \mapsto P^*(\ph_{q,t})$ is continuously differentiable on $R_\eta(I_Q)$, and that $\ph_{q,t}$ has an equilibrium state $\nu_{q,t}$ for every $(q,t)\in R_\eta(I_Q)$.  Then for every $\alpha\in(\alpha_2,\alpha_1)=(-D^-\Td(q_2),-D^+\Td(q_1))$ there exists a sequence $(q_n,t_n)\to (q,\Td(q))$ such that each $\ph_{q_n,t_n}$ has an ergodic equilibrium state $\nu_n$ satisfying~\eqref{eqn:alpha-integralD}.
\end{proposition}

As mentioned after the statement of Theorem~\ref{thm:dimension}, we can do away with the talk of sequences of potentials and measures in Propositions~\ref{prop:concaveD} and~\ref{prop:alpha-intervalD} if each $\ph_q$ has an equilibrium state $\nu_q$ with $\lambda(\nu_q)>0$ and if $\Td$ is $\CCC^r$ on $(q_1,q_2)$.  The proof in this case goes just like the proof we carry out below.

Before proceeding to the proof itself, we pause to collect pertinent results on the relationship between pointwise dimension, local entropy, and the Lyapunov exponent.  Given an ergodic measure $\nu\in \Mfe(X)$, the Lyapunov exponent $\lambda(x) = (\log a)^+(x)$ exists and is constant $\nu$-a.e.\ as a consequence of Birkhoff's ergodic theorem.  The analogous result for the local entropy $h_\nu(x)$ was proved by Brin and Katok~\cite{BK83}.  The following proposition shows (among other things) that together, these imply exactness of the measure $\nu$ when the map $f$ is conformal.

\begin{proposition}\label{prop:localdim}
Let $f\colon X\to X$ be continuous and conformal with continuous non-vanishing factor $a(x)$, and fix $\nu\in\Mf(X)$.  Suppose that the local entropy $h_\nu(x)$ and Lyapunov exponent $\lambda(x)$ both exist at some $x\in X$.  If $\lambda(x)>0$, then the pointwise dimension $d_\nu(x)$ also exists, and
\begin{equation}\label{eqn:localdim}
d_\nu(x) = \lim_{n\to\infty} \frac{-\log \nu(B(x,n,\delta))}{S_n\log a(x)} = \frac{ h_\nu(x) }{\lambda(x)}.
\end{equation}
If $\lambda(x) = 0$ and $h_\nu(x) > 0$, then $d_\nu(x)$ exists and is equal to $+\infty$.
\end{proposition}
\begin{proof}
Fix $\eps>0$; if $\lambda(x) > 0$, choose $\eps < \lambda(x)$.  Since $\lambda(x)$ exists we may apply Lemma~\ref{lem:well-behaved} and obtain $\delta = \delta(\eps)>0$ and $\eta = \eta(x) > 0$ such that~\eqref{eqn:diamball} holds for all $n\in\NN$, and hence writing
\begin{equation}\label{eqn:rnsn}
r_n = \eta \delta e^{-n(\lambda_n(x) + \eps)}, \qquad
s_n = \delta e^{-n(\lambda_n(x)-\eps)},
\end{equation}
we have
\begin{equation}\label{eqn:nuballs}
\nu(B(x,r_n)) \leq \nu(B(x,n,\delta)) \leq \nu(B(x,s_n)).
\end{equation}
Observe that
\begin{equation}\label{eqn:logrn}
\log r_n = \log (\eta \delta) - S_n \log a(x) - n\eps,
\end{equation}
and that furthermore,
\begin{equation}\label{eqn:logrnratio}
\begin{aligned}
\frac{\log r_{n+1}}{\log r_n} &= \frac{\log (\eta \delta) - S_{n+1}\log a(x) - (n+1)\eps}{\log (\eta \delta) - S_n \log a(x) - n\eps} \\
&= 1 - \frac{\eps + \log a(f^n(x))}{\log(\eta \delta) - S_n \log a(x) - n\eps}.
\end{aligned}
\end{equation}
Observe that the numerator is uniformly bounded, and that if $\lambda(x) > 0$, the denominator goes to $-\infty$ by the assumption that $\eps < \lambda(x)$, while if $\lambda(x) = 0$, the denominator goes to $-\infty$ because $\left\lvert \frac 1n S_n \log a(x)\right\rvert < \frac{\eps}2$ for all sufficiently large $n$.  It follows that the ratio in~\eqref{eqn:logrnratio} converges to $1$, and a similar result holds for $s_n$.  The same argument shows that $r_n \to 0$ for all values of $\lambda(x)$, while $s_n\to 0$ provided $\lambda(x)>0$.

For future reference, we point out that everything up to this point also holds if $x\in \BBB$ and $\llambda(x) > 0$.

Now suppose that $\lambda(x)>0$.  It follows that
\begin{equation}\label{eqn:ratio2}
\lim_{n\to\infty} \frac{-\log r_n}{S_n \log a(x)} = 
\lim_{n\to\infty} \left(1 + \frac{n\eps - \log (\eta\delta)}{S_n\log a(x)}\right) = 1 + \frac{\eps}{\lambda(x)}.
\end{equation}
and we see from the first inequality in~\eqref{eqn:nuballs} that
\[
\frac{\log \nu(B(x,r_n))}{\log r_n} \left(\frac{-\log r_n}{S_n\log a(x)}\right)
 \geq \frac{-\log \nu(B(x,n,\delta))}{S_n \log a(x)},
\]
where we observe that the quantity on the right is exactly the quantity that appears in~\eqref{eqn:localdim}.  Letting $n$ tend to infinity, this yields
\begin{equation}\label{eqn:ndimgeq}
\llim_{n\to\infty} \frac{\log \nu(B(x,r_n))}{\log r_n}
\left(1+\frac{\eps}{\lambda(x)}\right) \geq \frac{h_\nu(x)}{\lambda(x)}.
\end{equation}
Now given an arbitrary $r>0$, let $n$ be such that $r_n \leq r \leq r_{n-1}$; it follows that
\[
\frac{\log \nu(B(x,r))}{\log r} \geq \frac{\log \nu(B(x,r_n))}{\log r_{n-1}}
= \frac{\log \nu(B(x,r_n))}{\log r_n} \frac{\log r_n}{\log r_{n-1}},
\]
and since $\log r_n / \log r_{n-1} \to 1$, we may let $r$ tend to $0$ to obtain
\[
\lowd_\nu(x) \left( 1+\frac{\eps}{\lambda(x)} \right) \geq \frac{h_\nu(x)}{\lambda(x)}.
\]
Since $\eps>0$ was arbitrary, this gives
\[
\lowd_\nu(x) \geq \frac{h_\nu(x)}{\lambda(x)}.
\]
Using similar estimates on $s_n$, we obtain the upper bound
\[
\uppd_\nu(x) \leq \frac{h_\nu(x)}{\lambda(x)},
\]
which implies~\eqref{eqn:localdim}.

It only remains to consider the case $\lambda(x) = 0$.  We first observe that in this case we can choose $N$ sufficiently large that $|S_n\log a(x) - \log (\eta\delta)| < n\eps$ for all $n\geq N$, and hence $0 > \log r_n > -2n\eps$.  Then the first inequality in~\eqref{eqn:nuballs} gives
\[
\frac{\log \nu(B(x,r_n))}{\log r_n} > -\frac{1}{2n\eps} \log \nu(B(x,n,\delta)),
\]
and taking the limit as $n\to\infty$ gives
\[
\ld_\nu(x) > \frac{h_\nu(x)}{2\eps},
\]
just as above.  Since $\eps>0$ was arbitrary, we have $d_\nu(x)=+\infty$.
\end{proof}

The following corollaries of Proposition~\ref{prop:localdim} are easily proved by considering generic points for the measure $\nu$.

\begin{corollary}\label{cor:youngs}
Let $f\colon X\to X$ be continuous and conformal with continuous non-vanishing factor $a(x)$, and fix $\nu\in \Mf(X)$ with $\lambda(\nu)>0$.  Then $\dim_H \nu = h(\nu) / \lambda(\nu)$.
\end{corollary}

\begin{corollary}\label{cor:ptwisedim}
Let $f\colon X\to X$ be continuous and conformal with continuous non-vanishing factor $a(x)$, and fix $\mu,\nu\in \Mf(X)$.  Suppose that $\lambda(\nu)>0$, and let $\alpha\in \RR$ be given by
\[
\alpha = \frac{\int h_\mu(x) \,d\nu(x)}{\lambda(\nu)}.
\]
Then $\nu(\Kad(\mu)) =1$, where $\Kad(\mu)$ is the set of points $x\in X$ for which $d_\mu(x)=\alpha$.
\end{corollary}

Given a little more information about $X$, we can also say something about measures with zero Lyapunov exponent.

\begin{corollary}\label{cor:finitedimh}
Let $f\colon X\to X$ be continuous and conformal with continuous non-vanishing factor $a(x)$, and suppose that $\dim_H X<\infty$.  Then any $\nu\in \Mf(X)$ with $\lambda(\nu)=0$ must have $h(\nu)=0$ as well.
\end{corollary}
\begin{proof}
First suppose that $\nu$ is ergodic and that $h(\nu)>0$.  Then by Birkhoff's ergodic theorem and the Brin--Katok entropy formula, there exists a set $Y\subset X$ such that $\nu(Y)=1$ and for every $x\in Y$, we have $\lambda(x) = 0$ and $h_\nu(x) = h(\nu) > 0$.  It follows from Proposition~\ref{prop:localdim} that $d_\nu(x) = +\infty$, and hence 
\[
\dim_H X \geq \dim_H \nu = +\infty,
\]
which contradicts the assumption in Theorem~\ref{thm:dimension} that $\dim_H X<\infty$.
\end{proof}

A converse of sorts to Proposition~\ref{prop:localdim} is given by the following, which addresses the case where $d_\mu(x)$ exists even though $h_\mu(x)$ and $\lambda(x)$ may not.  We exclude points lying in $\ZZZ = \ZZZ(\mu)$.

\begin{proposition}\label{prop:localdim2}
Let $f\colon X\to X$ be continuous and conformal with continuous non-vanishing factor $a(x)$, and fix $\mu\in\Mf(X)$.  Suppose that the pointwise dimension $d_\mu(x)$ exists at some point $x\in X' \cap \BB$ and is equal to $\alpha$.  Then although the local entropy and Lyapunov exponent may not exist at $x$, the ratio of the pre-limit quantities still converges; in particular, we have
\begin{equation}\label{eqn:localdim2}
\lim_{n\to\infty} \frac{-\log \mu(B(x,n,\delta))}{S_n\log a(x)} = \alpha = d_\mu(x)
\end{equation}
whenever $\llambda(x) > 0$, and $\alpha=\infty$ if $\llambda(x) = 0$.
\end{proposition}
\begin{proof}
We deal first with the case $\llambda(x) = 0$.  In this case, there exists an increasing sequence $n_k$ such that
\[
\frac 1{n_k} S_{n_k} \log a(x) \to 0,
\]
and since $x\notin \ZZZ$, there exists $\delta_0>0$ such that 
\[
\gamma(\delta) := \llim_{k\to\infty} -\frac 1{n_k} \log \mu(B(x,n_k,\delta)) > \gamma(\delta_0) > 0
\]
for any $0<\delta<\delta_0$.

Fix $\eps>0$.  Because $x\in \BBB$, we may apply Lemma~\ref{lem:well-behaved} to get $r_n$ as in~\eqref{eqn:rnsn} for which~\eqref{eqn:nuballs} holds for $\mu$, and we have $r_{n_k}\to 0$ just as in the proof of Proposition~\ref{prop:localdim}.  In particular, for all sufficiently large $k$,~\eqref{eqn:nuballs} gives
\[
\frac{\log \mu(B(x,r_{n_k}))}{\log r_{n_k}} > -\frac{1}{2n_k\eps} \log \mu(B(x,n_k,\delta)),
\]
and it follows that
\[
\alpha = \lim_{k\to\infty} \frac{\log \mu(B(x,r_{n_k}))}{\log r_{n_k}} \geq \frac{\gamma(\delta_0)}{2\eps}.
\]
Since $\eps>0$ was arbitrary, we see that $\alpha=\infty$.  (Observe that since the hypothesis of the proposition tells us that $d_\mu(x)$ exists, it suffices to obtain $\ld_\mu(x)=\infty$, as we do here.)

We turn now to the case $\llambda(x)>0$.  As remarked in the proof of Proposition~\ref{prop:localdim}, the computations at the beginning of that proof are valid here as well; everything up to but not including~\eqref{eqn:ratio2} works in the present setting. \eqref{eqn:ratio2} is replaced by the following inequality:
\[
\ulim_{n\to\infty} \frac{-\log r_n}{S_n\log a(x)} \leq 1 + \frac{\eps}{\llambda(x)}.
\]
Thus we have the following in place of~\eqref{eqn:ndimgeq}:
\begin{align*}
d_\mu(x) \left( 1 + \frac{\eps}{\llambda(x)}\right) &= 
\lim_{n\to\infty} \frac{\log \mu(B(x,r_n))}{\log r_n} \left( 1 + \frac{\eps}{\llambda(x)} \right) \\
&\geq \ulim_{n\to\infty} \frac{-\log \mu(B(x,n,\delta))}{S_n\log a(x)}.
\end{align*}
Similar computations with $s_n$ give
\[
d_\mu(x) \left( 1 - \frac{\eps}{\llambda(x)} \right) \leq \llim_{n\to\infty}\frac{-\log \mu(B(x,n,\delta))}{S_n\log a(x)},
\]
and since $\eps>0$ was arbitrary, this suffices to prove~\eqref{eqn:localdim2}.
\end{proof}

\begin{proof}[Proof of Theorem~\ref{thm:dimension}]
We prove part I of the theorem by establishing the following two inequalities:
\begin{align}
\Td &\leq \DL, \label{eqn:TleqDL} \\
\Td &\geq \DL. \label{eqn:DleqTL}
\end{align}
We begin by proving~\eqref{eqn:TleqDL}.  First, observe that we may have $\Td(q)=+\infty$ for some values of $q$.  Suppose that this is the case for some $q\in \RR$; then for any sequence $t_n\to+\infty$, we have $P^*(\ph_{q,t_n}) > 0$ for all $n$, and hence there exists a sequence of ergodic $f$-invariant measures $\nu_n$ such that
\begin{equation}\label{eqn:pospress}
h(\nu_n) + q\int \ph_1\,d\nu_n - t_n \lambda(\nu_n) > 0.
\end{equation}
Now there are two possibilities.

\emph{Case 1.}  $\lambda(\nu_n) > 0$ for all $n$.  In this case we obtain
\[
\frac{h(\nu_n)}{\lambda(\nu_n)} + q\frac{\int \ph_1\,d\nu_n}{\lambda(\nu_n)} > t_n.
\]
Applying Corollary~\ref{cor:youngs}, we see that the first term is equal to $\dim_H \nu$; furthermore,  Corollary~\ref{cor:ptwisedim} together with the weak Gibbs property of $\mu$ gives $\nu_n(K_{\alpha_n}^\DDD) = 1$, where $\alpha_n = \int\ph_1\,d\nu_n / \lambda(\nu_n)$.  Consequently, we have
\[
\DDD(\alpha_n) + q\alpha_n \geq \dim_H \nu_n + q\alpha_n > t_n,
\]
and it follows that $\DL(q) = \sup_{\alpha\in \RR} (\Da + q\alpha) = +\infty$.

\emph{Case 2.}  There exists $n$ such that $\lambda(\nu_n) = 0$.  Then Corollary~\ref{cor:finitedimh} implies that $h(\nu_n) = 0$ as well, and~\eqref{eqn:pospress} gives us that $q\int \ph_1\,d\nu_n > 0$.  If $q\geq 0$, this is impossible, since $\int \ph_1 \,d\nu \leq 0$ for all $\nu\in \Mf(X)$.  If $q<0$, this implies that $\int \ph_1\,d\nu_n < 0$, and hence $\nu_n(\ZZZ) = 0$.  Now for $\nu_n$-a.e.\ $x\in X$, we may apply Proposition~\ref{prop:localdim2} to obtain $d_\mu(x) = +\infty$.  It follows that $\nu_n(K_\infty^\DDD) = 1$ and $K_\infty^\DDD \cap X' \neq \emptyset$, and we once again have $\DL(q) = +\infty$.

Having dealt with the case where $\Td(q)=+\infty$, we now turn our attention to the case where $\Td(q)$ is finite.  Given $t<\Td(q)$, we observe that any measure $\nu$ with $h(\nu) + \int \ph_{q,t}\,d\nu > 0$ must also satisfy $\lambda(\nu) > 0$, otherwise we would have $\Td(q) = +\infty$.  It follows that
\[
P^*(\ph_{q,t}) = \sup \left\{ h(\nu) + \int \ph_{q,t} \,d\nu \,\Big|\, \nu\in \Mfe(X), \lambda(\nu)>0 \right\}.
\]
Given $\alpha,\lambda \geq 0$, consider the following set:
\[
Z_{\alpha,\lambda} = \{ x\in X \mid \ph_1^+(x) = -\alpha\lambda, \lambda(x) = \lambda \}.
\]
Every ergodic measure $\nu$ is supported on some $Z_{\alpha,\lambda}$, and so we have
\[
0 < P^*(\ph_{q,t}) = \sup_{\alpha\geq 0} \sup_{\lambda>0} 
\sup \left\{ h(\nu) + \int \ph_{q,t} \,d\nu \,\Big|\, \nu\in \Mfe(X), \nu(Z_{\alpha,\lambda}) = 1 \right\}.
\]
It follows that there exists some $\alpha$, $\lambda$, and $\nu$ for which $\nu(Z_{\alpha,\lambda}) =1$ and
\[
h(\nu) + q \int \ph_1 \,d\nu - t\lambda(\nu) > 0.
\]
Applying Corollaries~\ref{cor:youngs} and~\ref{cor:ptwisedim} as before, we see that $\nu(\Kad)=1$ and
\[
(\dim_H \nu - q \alpha - t) \lambda > 0,
\]
which immediately yields
\[
t < \Da - q\alpha.
\]
Since $t<\Td(q)$ was arbitrary, this proves~\eqref{eqn:TleqDL}.

In order to show~\eqref{eqn:DleqTL}, we show that
\begin{equation}\label{eqn:TgeqD}
\Td(q) \geq \Da - q\alpha
\end{equation}
for every $q\in \RR$ and $\alpha \in \RR$.  (Observe that Proposition~\ref{prop:emptyD} deals with the case $\alpha=\infty$.)

Recall from~\eqref{eqn:Td} that
\[
\Td(q) = \inf\{t\in \RR\mid P^*(q\ph_1 - t\log a)\geq 0\} = \sup\{t\in \RR\mid P^*(q\ph_1 - t\log a) > 0\},
\] 
and so to establish~\eqref{eqn:TgeqD} (and hence~\eqref{eqn:DleqTL}), it suffices to show that $P^*(q\ph_1 - t\log a) > 0$ for every $t < \Da - q\alpha$.

To this end, fix $q, t\in \RR$ such that $t + q\alpha < \Da = \dim_H \Kad$.  We will build a measure $\nu$ such that
\begin{equation}\label{eqn:largefree}
h(\nu) + \int q\ph_1 \,d\nu - t\lambda(\nu) > 0,
\end{equation}
which will suffice to complete the proof of~\eqref{eqn:DleqTL}, by the above remarks.  Observe that since $\dim_H \ZZZ = 0$, we have
\[
\dim_H (\Kad \setminus \ZZZ) = \dim_H \Kad > t+q\alpha;
\]
furthermore, it follows from Proposition~\ref{prop:localdim2} that $\llambda(x) > 0$ for every $x\in \Kad \setminus \ZZZ$, and so we may apply~\cite[Theorem 2.1]{vC09a} and obtain
\[
P_{\Kad \setminus \ZZZ}(-(t+q\alpha)\log a) > 0,
\]
where $P_Z$ is the (Carath\'eodory dimension) topological pressure on $Z$.  Fix $\gamma>0$ small enough that we have
\begin{equation}\label{eqn:gammasmall}
P_{\Kad \setminus \ZZZ} (-(t+q\alpha) \log a) - \gamma > \gamma > 0.
\end{equation}

Now define a family of sets as in~\eqref{eqn:Faen}: for every $\eps>0$ and $n\in \NN$, consider the set
\begin{equation}\label{eqn:Gaen}
\Gaen = \left\{ x\in X \,\Big|\, \abs{\frac {-S_n\ph_1(x)}{S_n \log a(x)} - \alpha}
\leq \eps \text{ and $S_n\log a(x) > 0$ for all } n\geq N \right\}.
\end{equation}
We will also make use of the following sets:
\begin{equation}\label{eqn:Gae}
\Gae = \bigcup_{N\in\NN} \Gaen.
\end{equation}

Applying Proposition~\ref{prop:localdim2} and using the fact that $\mu$ is a weak Gibbs measure for $\ph$, we see that for every $x\in \Kad \setminus \ZZZ$,
\[
\alpha = d_\mu(x) = \lim_{n\to\infty} \frac{- S_n\ph_1(x)}{S_n\log a(x)}.
\]
Since $\llambda(x) > 0$ for every $x\in \Kad \setminus \ZZZ$, this implies $\Kad \setminus \ZZZ \subset \Gae$ for every $\eps>0$.  In particular, this implies that
\[
P_{\Kad \setminus \ZZZ}(-(t+q\alpha)\log a) \leq P_\Gae(-(t+q\alpha)\log a) =\sup_{N\in \NN} P_\Gaen(-(t+q\alpha)\log a),
\]
and so there exists $N\in \NN$ such that
\begin{equation}\label{eqn:gammasmall2}
P_\Gaen(-(t+q\alpha)\log a) - \gamma > \gamma > 0.
\end{equation}
Now we can apply the general inequality~\cite[(11.9)]{yP98} to obtain
\begin{equation}\label{eqn:gammasmall3}
\lP_{\Gaen}(-(t+q\alpha)\log a) - \gamma > \gamma > 0.
\end{equation}

Let $\psi(x) = \ph_1(x) + (\alpha + \eps)\log a(x)$, and observe that for every $x\in \Gaen$ and $n\geq N$, we have
\[
\abs{-S_n \ph_1(x) - \alpha S_n \log a(x)} \leq \eps S_n \log a(x),
\]
which gives
\[
-S_n \ph_1(x) \leq (\alpha +\eps) S_n \log a(x),
\]
and in particular, $S_n \psi(x) \geq 0$.  We may now apply Lemma~\ref{lem:buildmeasure} with $\psi = \ph_1 + (\alpha+\eps)\log a$, $\zeta = -(t+q\alpha) \log a$, $Z = \Gaen$, and $\gamma$ as before, to obtain a measure $\nu\in \Mf(X)$ with the following properties:
\begin{align}
\label{eqn:int}
\int \ph_1 \,d\nu + (\alpha + \eps) \lambda(\nu) &\geq 0, \\
\label{eqn:free}
h(\nu) - (t+q\alpha) \lambda(\nu) &\geq \lP_{\Gaen}(-(t+q\alpha)\log a) - \gamma > \gamma > 0.
\end{align}
If $q\geq 0$, then multiplying~\eqref{eqn:int} by $q$ yields
\[
\int q\ph_1 \,d\nu + (q\alpha + q\eps) \lambda(\nu) \geq 0,
\]
and adding this to~\eqref{eqn:free} yields
\[
h(\nu) + \int q\ph_1 \,d\nu - t\lambda(\nu) \geq \gamma - q\eps\lambda(\nu).
\]
We can choose $\eps>0$ small enough such that $\gamma > q\eps\lambda(\nu)$ for any invariant measure $\nu$, and this establishes~\eqref{eqn:largefree}.

For $q\leq 0$, we do a similar computation with $\psi = \ph_1 + (\alpha - \eps)\log a$.
\end{proof}

We now proceed to the proof of Part II.

\begin{proof}[Proof of Proposition~\ref{prop:emptyD}]
Suppose there exists $x\in \Kad \setminus \ZZZ$, and let $n_k$ be a subsequence such that the empirical measures $\mu_{x, n_k}$ converge to an invariant measure $\nu$.  Then $\lambda(\nu)>0$ (otherwise $\alpha=\infty$ or $x\in \ZZZ$) and $-\int \ph_1 \,d\nu = \alpha \int \log a\,d\nu$ (by Proposition~\ref{prop:localdim2} and weak* convergence).  It follows that
\begin{align*}
P^*(q\ph_1 - t\log a) &\geq h(\nu) + \int q\ph_1 \,d\nu - \int t\log a \,d\nu \\
&\geq -\lambda(\nu) (q\alpha + t)
\end{align*}
for every $q,t\in \RR$.  In particular, if $P^*(\ph_{q,t}) \leq 0$, then $q\alpha + t\geq 0$, hence $t\geq -q\alpha$.  This holds for all $t\geq \Td(q)$, and consequently $\Td(q) \geq -q\alpha$ as well.

As for the case $\alpha=\infty$, we use the above construction and Corollary~\ref{cor:finitedimh} to obtain $\nu\in \Mf(X)$ with $\lambda(\nu) = h(\nu) = 0$.  Furthermore, since $x\in X'$, we have $\int \ph_1 \,d\nu < 0$, and it follows immediately that $P^*(\ph_{q,t}) > 0$ for all $q < 0$ and $t\in \RR$, hence $\Td(q) = +\infty$ for all $q<0$.
\end{proof}

It only remains to prove the propositions implying Part III.

\begin{proof}[Proof of Proposition~\ref{prop:concaveD}]
It follows from Corollary~\ref{cor:ptwisedim} and the weak Gibbs property of $\mu$ that $\nu_n(\Kad) = 1$ for all $n$.  Furthermore, from the assumption that $t_n \leq \Td(q_n)$, we have
\[
0 \leq P^*(\ph_{q_n,t_n}) = h(\nu_n) + q_n\int \ph_1 \,d\nu_n - t_n \lambda(\nu_n)
= h(\nu_n) - q_n \alpha \lambda(\nu_n) - t_n \lambda(\nu_n),
\]
and applying Corollary~\ref{cor:youngs} (using the assumption that $\lambda(\nu_n)>0$) gives
\[
\dim_H \nu_n \geq q_n \alpha + t_n.
\]
Since $\nu_n(\Kad) = 1$, this in turn implies
\[
\Da \geq q_n \alpha + t_n,
\]
and taking the limit as $n\to\infty$ yields
\[
\Da \geq q\alpha + \Td(q) \geq \TdL(\alpha).\qedhere
\]
\end{proof}

\begin{proof}[Proof of Proposition~\ref{prop:alpha-intervalD}]
As before, it follows from the finiteness of $\Td(q)$ that $\frac{\di}{\di_t} P^*(\ph_{q,t}) = -\lambda(\nu_{q,t}) < 0$ for all $(q,t)\in R_\eta(I_Q)$, and consequently (assuming $n$ is large enough) we may apply the Implicit Function Theorem to obtain a continuously differentiable function $T_n\colon (q_1,q_2) \to \RR$ such that $(q,T_n(q)) \in R_\eta(I_Q)$ for all $q$, and such that
\[
P^*(\ph_{q,T_n(q)}) = \frac 1n.
\]
Furthermore, we have
\[
\lim_{n\to\infty} D^+ T_n(q_1) = D^+ \Td(q_1),\qquad
\lim_{n\to\infty} D^- T_n(q_2) = D^- \Td(q_2),
\]
so for every $\alpha$ as in the statement of the proposition, and for all sufficiently large $n$, we have
\[
-D^- T_n(q_2) < \alpha < -D^+ T_n(q_1).
\]
In particular, by the Intermediate Value Theorem, there exists $q_n$ such that $T_n'(q_n) = -\alpha$.  Let $t_n = T_n(q_n)$; then by passing to a subsequence if necessary, we may assume that $(q_n,t_n) \to (q,\Td(q))$ for some $q\in I_Q$.  Let $\nu_n$ be an ergodic equilibrium state for $\ph_{q_n,t_n}$; because $P^*(\ph_{q_n,t_n}) >0$ and $\Td(q_n) < \infty$, we have $\lambda(\nu_n) > 0$.

Finally, we observe that since $P^*(\ph_{q,t})$ is constant along the curve $(q,T_n(q))$, we have
\begin{align*}
0 = \frac{d}{dq} P^*(\ph_{q,T_n(q)})|_{q_n} &= \frac{\di}{\di q} P^*(\ph_{q,t})|_{(q_n,t_n)}  
+ T_n'(q_n) \frac{\di}{\di t} P^*(\ph_{q,t})|_{(q_n,t_n)} \\
&= \int \ph_1 \,d\nu_n + \alpha \lambda(\nu_n),
\end{align*}
and hence $\nu_n$ satisfies~\eqref{eqn:alpha-integralD}.
\end{proof}

\bibliographystyle{alpha}
\bibliography{abbrev,books,thermodynamic,multifractal}

\end{document}